\documentclass[11pt, reqno, DIV= 11, oneside]{amsart}

\usepackage{typearea}
\usepackage{lineno,hyperref}
\usepackage{amssymb} % for mathbb
\usepackage{dsfont} % for mathds
\usepackage{amsmath} % for align, eqref
\usepackage{xfrac} % for sfrac
\usepackage{cases} % for subnumcases
\usepackage{xcolor} % for textcolor
\usepackage{subcaption} % for subfig
\usepackage{enumitem} % for enumerate
\usepackage{mathtools} % for dcases
\usepackage[ruled,vlined, linesnumberedhidden]{algorithm2e}
\SetKwInOut{Input}{Input}
\usepackage{amsthm}

\newtheorem{theorem}{Theorem}[section]
\newtheorem{definition}[theorem]{Definition}
\newtheorem{remark}[theorem]{Remark}
\newtheorem{lemma}[theorem]{Lemma}

\newtheorem{assumption}{Assumption}
\newtheorem{corollary}[theorem]{Corollary}

\hypersetup{
	colorlinks,
	citecolor=red!70!black,
	filecolor=black,
	linkcolor=blue!70!black,
	urlcolor=blue!70!black
}

\calclayout

\setenumerate[1]{label={\normalfont \textbf{\arabic*.}}}

\title[Nodal profile control for networks of GEB]{Nodal profile control for networks of geometrically exact beams}

\author{Günter Leugering}

\author{Charlotte Rodriguez}

\author{Yue Wang}

\thanks{Accepted in \textit{Journal de Mathématiques Pures et Appliquées},  \texttt{DOI}: \href{https://doi.org/10.1016/j.matpur.2021.07.007}{\texttt{10.1016/j.matpur.2021.07.007}}.\\
AMS subject classification. 35L50, 35R02, 93B05, 93C20, 35Q74, 74K10.\\
\noindent \textit{Keywords. Geometrically exact beam, networks, nodal profile controllability, well-posedness.}\\
\noindent \textbf{Funding:} This project is supported by the Deutsche Forschungsgemeinschaft DFG L595/31-1 and the European Union’s Horizon 2020 research and innovation programme under the Marie Sklodowska-Curie grant agreement No.765579-ConFlex.\\
\textit{Email adresses}: \texttt{guenter.leugering@fau.de, charlotte.rodriguez@fau.de, yue.wang@fau.de}}

\date{\today}

\begin{document}

\maketitle

\begin{center}\small 
\textit{Chair of Applied Mathematics 2, Department of Mathematics, \\
Friedrich-Alexander-Universit\"at Erlangen-N\"urnberg,\\
Cauerstr. 11, 91058 Erlangen, Germany}
\end{center}

\begin{abstract}
In this work, we consider networks of so-called geometrically exact beams, namely, shearable beams that may undergo large motions. The corresponding mathematical model, commonly written in terms of displacements and rotations expressed in a fixed basis (Geometrically Exact Beam model, or GEB), has a quasilinear governing system. However, the model may also be written in terms of intrinsic variables expressed in a moving basis attached to the beam (Intrinsic GEB model, or IGEB) and while the number of equations is then doubled, the latter model has the advantage of being of first-order, hyperbolic and only semilinear. First, for any network, we show the existence and uniqueness of semi-global in time classical solutions to the IGEB model (i.e., for arbitrarily large time intervals, provided that the data are small enough). Then, for a specific network containing a cycle, we address the problem of local exact controllability of nodal profiles for the IGEB model -- we steer the solution to satisfy given profiles at one of the multiple nodes by means of controls applied at the simple nodes -- by using the constructive method of Zhuang, Leugering and Li [\emph{Exact boundary controllability of nodal profile for Saint-Venant system on a network with loops}, in J. Math. Pures Appl., 2018]. Afterwards, for any network, we show that the existence of a unique classical solution to the IGEB network implies the same for the corresponding GEB network, by using that these two models are related by a nonlinear transformation. In particular, this allows us to give corresponding existence, uniqueness and controllability results for the GEB network.
\end{abstract}

\tableofcontents

\addtocontents{toc}{\protect\setcounter{tocdepth}{1}}

%\linenumbers

%%%%%%%%%%%%%%%%%%%%%%%%%%%%%%%%%%%%%%%%%
%%%%%%%%%%%%%%% MAIN %%%%%%%%%%%%%%%%%%%%
%%%%%%%%%%%%%%%%%%%%%%%%%%%%%%%%%%%%%%%%%

\section{Introduction}
\label{sec:intro}

\noindent \textbf{Nodal profile controllability.}
The problem of \emph{nodal profile controllability} of partial differential equations on networks refers to the task of steering the solution thereof to prescribed profiles on specific nodes. Formally speaking, this amounts to saying that said solution should be controlled to given time-dependent functions (called \emph{nodal profiles}) over certain time intervals by means of controls actuating at one or several other nodes.
This is in contrast to the classical question of exact controllability, wherein one seeks to steer the state, at a certain time, to a given final state on the entire network.
The nodes with prescribed profiles are then called \emph{charged nodes} \cite{YWang2019partialNP} (or \emph{object-nodes} \cite{Zhuang2021}), while the nodes at which the controls are applied are the \emph{controlled nodes} \cite{YWang2019partialNP} (or \emph{control nodes} \cite{Zhuang2021}).

The notion of exact boundary controllability of nodal profiles was, to our knowledge, first introduced by Gugat, Herty and Schleper in \cite{gugat10}, motivated by applications in the context of gas transport through pipelines networks. 
Therein, consumers are located at the endpoints of the network and the nodal profiles represent the consumer satisfaction, and the former are sought to be attained by the flow which is controlled by means of a number of compressors actuating at several nodes.

\medskip

\noindent Motivated by the abundant practical relevance of such control problems, Tatsien Li and coauthors generalized the aforementioned results to one-dimensional first-order quasilinear hyperbolic systems with nonlinear boundary conditions \cite{gu2011, li2010nodal, li2016book}. 
For results on the wave equation on a tree-shaped network with a general topology or the unsteady flow in open canals, we refer the reader to \cite{kw2011, kw2014, YWang2019partialNP} and \cite{gu2013}, respectively. 

Whilst the exact-controllability of the Saint-Venant equations on networks with cycles is not true in general \cite{LLS, li2010no}, for certain networks with cycles, the exact nodal profile controllability can be shown by means of a so-called \emph{cut-off} method \cite{Zhuang2021, Zhuang2018}.
In this regard, in line with intuition, the concept of nodal profile controllability is weaker than that of exact controllability.
Hence, when considering a system defined on a network with cycles, a situation which is encountered in many practical applications, the nodal profile control problem is a rather meaningful and feasible goal to attain.

\medskip

\noindent
The method used by Li et al. to prove nodal profile controllability is \emph{constructive} in nature, in the sense that it relies on solving the equation forward in time and sidewise, to build a specific solution which achieves the desired goal, before evaluating the trace of this solution to obtain the desired controls. 
All this is done in the context of regular $C^1_{x,t}$ solutions for first-order systems, which are \emph{semi-global} in time -- this means that for any time $T>0$, and for small enough initial and boundary data, a unique solution exists at least until time $T>0$ --, a solution concept originating from \cite{LiJin2001_semiglob}.
In \cite{LiRao2002_cam, LiRao2003_sicon}, this notion of solution is used for proving local exact boundary controllability of one-dimensional quasilinear hyperbolic systems. In these works, a general framework for a constructive method is proposed, from which all subsequent constructive methods derive.
The cornerstone of Li's method is thus the proof of semi-global existence and uniqueness, and, in the case of networks, a thorough study of the transmission conditions at multiple nodes. As solving a sidewise problem entails exchanging the role of the spatial and time variables, 
this method fundamentally exploits the one-dimensional nature of the system (see also Remark \ref{rem:controllability_thm} \ref{subrem:sidewise}).

Very recently, in the context of the one-dimensional linear wave equation, the controllability of nodal profiles has also been studied in the context of less regular states and controls spaces, by using the duality between controllability and observability and showing an observability inequality. For star-shaped networks, one may see \cite{YWang2021_NP_HUM} where the sidewise and D’Alembert Formula is used, and for a single string one may see \cite{Sarac2021} which relies on sidewise energy estimates.

\medskip

\noindent \textbf{Geometrically exact beams.} 
Multi-link flexible structures such as large spacecraft structures, trusses, robot arms, solar panels, antennae \cite{chen_serial_EBbeams,  flotow_spacecraft, LLS} have found many applications in civil, mechanical and aerospace engineering.
The behavior of such structures is generally modeled by networks of interconnected beams.

In this article, we will address the problem of nodal profile controllability for \emph{networks of beams}, possibly with cycles, a problem which has not yet been considered in the literature. 
The network in question consists of $N$ beams, indexed by $i \in \{1, \ldots, N\}$, evolving in $\mathbb{R}^3$, which are mutually linked via rigid joints. 
The beams are assumed to be freely vibrating, meaning that external forces and moments, such as gravity or aerodynamic forces, have been set to zero.

\medskip

\noindent 
Nowadays, there is a growing interest in modern highly flexible light-weight structures -- for instance robotic arms \cite{grazioso2018robot}, flexible aircraft wings \cite{Palacios2010aero} or wind turbine blades \cite{Munoz2020, wang2014windturbine} -- which exhibit motions of large magnitude, not negligible in comparison to the overall dimensions of the object.
To capture such a behavior, one has to consider a beam model which is \emph{geometrically exact}, in the sense that the governing system presents nonlinearities in order to also represent large motions -- i.e., large displacements of the centerline and large rotations of the cross sections. 

This beam model, similarly to the more well-known Euler-Bernoulli and Timoshenko systems, is one dimensional with respect to the spatial variable $x$ and accounts for linear elastic material laws, meaning that the strains (which are the local changes in the shape of the material) are assumed to be small.
Models for geometrically exact beams account for shear deformation, similarly to the Timoshenko system. 
Moreover, the geometrical and material properties of the beam may vary along the beam (indeed, we will see that the coefficients of the system depend on $x$), and the material may be anisotropic.
As a matter of fact, the Euler-Bernoulli and Timoshenko systems can be derived from geometrically exact beam models under appropriate simplifying assumptions \cite[Section IV]{Artola2021damping}.

\medskip

\noindent We will see, in Subsection \ref{subsec:GEBmodels}, that the mathematical model for geometrically exact beams may be written in terms of the position of the centerline of the beam and the orientation of its cross sections, with respect to a fixed coordinate system. This is the commonly known \emph{Geometrically Exact Beam model}, or GEB, which originates from the work of Reissner \cite{reissner1981finite} and Simo \cite{simo1985finite}. The governing system is quasilinear, consisting of six equations. One may draw a parallel with the wave equation as the GEB model is of second order both in space and time.

On the  other hand, the mathematical model can also be written in terms of so-called \emph{intrinsic} variables -- namely, velocities and internal forces/moments, or equivalently velocities and strains -- expressed in a moving coordinate system attached to the beam.
This yields the \emph{Intrinsic Geometrically Exact Beam model}, or IGEB, which is due to Hodges \cite{hodges1990, hodges2003geometrically}. The governing system then counts twelve equations.
An interesting feature of the IGEB model is that it falls into the class of one-dimensional first-order hyperbolic systems and is moreover only semilinear. Therefore, from a mathematical perspective, one gains access to the broad literature which has been developed on such system -- see notably by Li and Yu \cite{Li_Duke85}, Bastin and Coron \cite{BC2016} -- beyond the context of beam models.

%Taking advantage of the good compromise between model fidelity and computational cost offered by the IGEB model, 
Due to its less compound nature, the IGEB formulation is used in aeroelastic modelling and engineering, notably in the context of very light-weight and slender aircraft aiming to remain airborne almost perpetually, and that consequently exhibit great flexibility \cite{Palacios2017modes, Palacios2011intrinsic, Palacios2010aero}; see also \cite{Artola2020aero, Artola2019mpc, Artola2021damping} where the authors additionally take into account structural damping.

\medskip

\noindent On another hand, as pointed out in \cite[Sec. 2.3.2]{weiss99}, one may see the GEB model and IGEB model as being related by a \emph{nonlinear transformation} (which we define in \eqref{eq:transfo}). In this work, we will keep track of this link between both models, studying mathematically the latter, and then deducing corresponding results for the GEB model.

As commonly done in solid mechanics, both the GEB and IGEB models are \emph{Lagrangian descriptions} of the beam (as opposed to the \emph{Eulerian description}), in the sense that the independent variable $x$ is attached to matter ($x$ is a label sticking to the particles of the beam's centerline throughout the deformation history) rather than being attached to an inertial frame of reference.

The IGEB model can also be seen as the beam dynamics being formulated in the \emph{Hamiltonian} framework in continuum mechanics (see notably \cite[Sections 5, 6]{Simo1988}), while the GEB model corresponds to the \emph{Lagrangian} framework.
Then, taking into account the interactions of the beam with its environment, one may study the IGEB model from the perspective of \emph{Port-Hamiltonian Systems} (see \cite{Maschke1992} for the finite dimension setting and \cite{Schaft2002} and \cite[Chapter 7]{Zwart2012bluebook} for the infinite dimensions setting), as in  \cite{Macchelli2007, Macchelli2009} and \cite[Section 4.3.2]{Macchelli2009book}. See also the case of the Timoshenko model in \cite{Macchelli2004Timo}.
%% timo: Mattioni2020Timo

\subsection{Our contributions}
In this article we consider the problem of nodal profile controllability in the context of a specific network of geometrically exact beams containing one cycle.
\textcolor{black}{Afterwards, the case of other networks, possibly containing several cycles, is discussed in Section \ref{sec:conclusion}: we give a few typical examples, together with a brief  algorithm (Algorithm \ref{algo:control}) to realize nodal profile controllability under some requirements.}

\textcolor{black}{Our main results will be given on IGEB networks (Theorem \ref{th:controllability}) and GEB networks (Corollary \ref{coro:controlGEB}) as follows.}
\begin{enumerate}
\item We first consider a general network of beams whose dynamics are given by the IGEB model (System \eqref{eq:syst_physical} below). We show, in Theorem \ref{th:existence}, that there exists a unique semi-global in time $C_{x,t}^1$ solution to \eqref{eq:syst_physical}.

This theorem is also a necessary step to show Theorem \ref{th:controllability}, namely, the local exact controllability of nodal profiles for System \eqref{eq:syst_physical}, in the special case of an A-shaped network (see Fig. \ref{subfig:AshapedNetwork}).
More precisely, we drive the solution to satisfy given profiles at one of the multiple nodes by controlling the internal forces and moments at the two simple nodes.

\item For a general network, via Theorem \ref{thm:solGEB}, we make the link between the IGEB network (System \eqref{eq:syst_physical}) and the corresponding system \eqref{eq:GEB_netw} in which the beams dynamics are given by the GEB model. 
More precisely, we show that the existence of a unique $C^1_{x,t}$ solution to \eqref{eq:syst_physical} implies that of a unique $C^2_{x,t}$ solution to \eqref{eq:GEB_netw}, provided that the data of both systems fulfill some compatibility conditions.

In particular, Theorem \ref{thm:solGEB}, permits to translate Theorems \ref{th:existence} and \ref{th:controllability} to corresponding results in terms of the GEB model \eqref{eq:GEB_netw}, which are Corollaries \ref{coro:wellposedGEB} and \ref{coro:controlGEB}, respectively.
\end{enumerate}

\subsection{Notation}
\label{subsec:notation}

Let $m, n\in \mathbb{N}$. Here, the identity and null matrices are denoted by $\mathbf{I}_n \in \mathbb{R}^{n \times n}$ and $\mathbf{0}_{n, m} \in \mathbb{R}^{n \times m}$, and we use the abbreviation $\mathbf{0}_{n} = \mathbf{0}_{n, n}$. The transpose of $M\in \mathbb{R}^{m\times n}$ is denoted by $M^\intercal$.
The symbol $\mathrm{diag}(\, \cdot \, , \ldots, \, \cdot \, )$ denotes a (block-)diagonal matrix composed of the arguments.
We denote by $\mathcal{S}_{++}^n$ the set of positive definite symmetric matrices in $\mathbb{R}^{n \times n}$.
The cross product between any $u, \zeta \in \mathbb{R}^3$ is denoted $u \times \zeta$, and we shall also write $\widehat{u} \,\zeta = u \times \zeta$, meaning that $\widehat{u}$ is the skew-symmetric matrix 
\begingroup % keep the change local
\setlength\arraycolsep{3pt}
\renewcommand*{\arraystretch}{0.9}
\begin{linenomath}
\begin{equation*}
\widehat{u} = \begin{bmatrix}
0 & -u_3 & u_2 \\
u_3 & 0 & -u_1 \\
-u_2 & u_1 & 0
\end{bmatrix}, 
\end{equation*}
\end{linenomath}
\endgroup
and for any skew-symmetric $\mathbf{u} \in \mathbb{R}^{3 \times 3}$, the vector $\mathrm{vec}(\mathbf{u}) \in \mathbb{R}^3$ is such that $\mathbf{u} = \widehat{\mathrm{vec}(\mathbf{u})}$. Finally, $\{e_\alpha\}_{\alpha=1}^3 = \{(1, 0, 0)^\intercal, (0, 1, 0)^\intercal, (0, 0, 1)^\intercal\}$ denotes the standard basis of $\mathbb{R}^3$.

\subsection{Outline}

In Section \ref{sec:model_results}, we present in more detail the GEB and IGEB models (Subsection \ref{subsec:GEBmodels}) before introducing the corresponding systems which give the dynamics of the beam network (Subsection \ref{subsec:network_systems}). Then, in Subsection \ref{subsec:main_results} we presents the main results of this article.

Section \ref{sec:exist} is concerned with the well-posedness of the network system \eqref{eq:syst_physical}: 
in Subsections \ref{subsec:hyperbolic} and \ref{subsec:change_var} we show that the system \eqref{eq:syst_physical} is hyperbolic and write it in Riemann invariants, we then study the transmission conditions for the diagonalized system in Subsection \ref{subsec:out_in_info}, and finally, we prove Theorem \ref{th:existence} in Subsection \ref{subsec:proof_exist}.

In Sections \ref{sec:controllability} and \ref{sec:invert_transfo}, we give the proofs of Theorems \ref{th:controllability} and \ref{thm:solGEB}, respectively.

\textcolor{black}{Then, in Section 6, we give generalized considerations on more involved networks, namely, with more than one cycles, or with prescribed profiles on several nodes.}

\section{The model and main results}
\label{sec:model_results}

As mentioned in the introduction, the beams' dynamics may be given from different points of view, that we specify in the following subsection.

\subsection{Dynamics of a geometrically exact beam}
\label{subsec:GEBmodels}

\begin{figure} \centering
\includegraphics[scale=0.7]{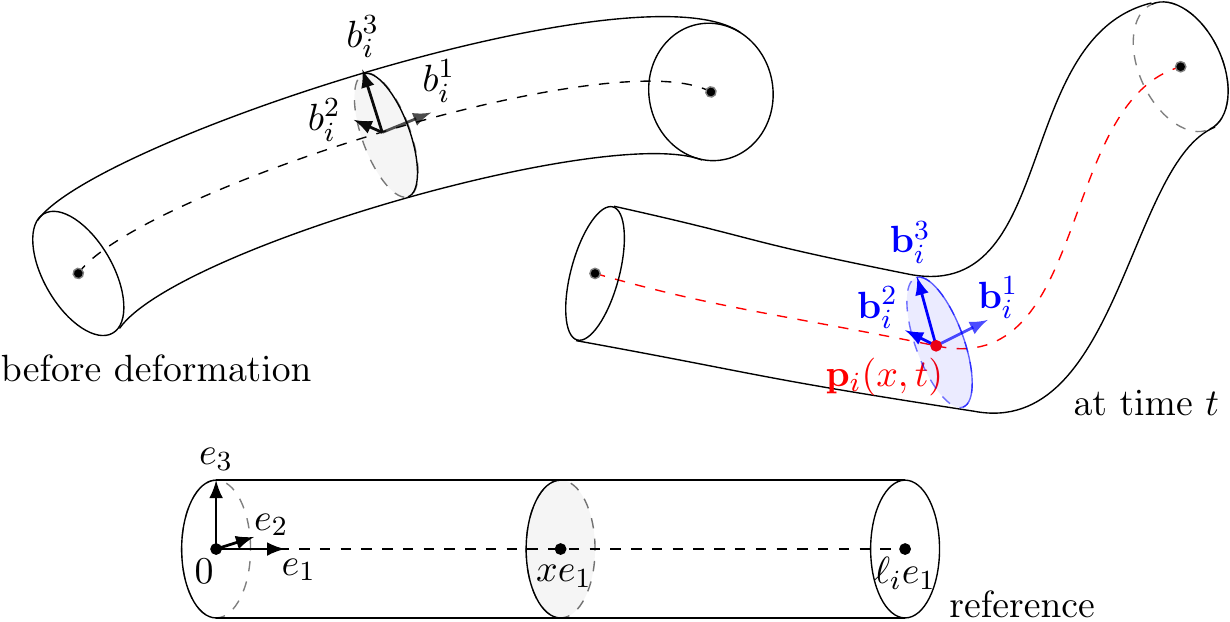}
\caption{Beam $i$ in a straight reference configuration, before deformation and at time $t$. Here, $\{b_i^\alpha\}_{\alpha=1}^3$ denote the columns of $R_i$.}
\label{fig:beam_netw}
\end{figure}
 
Let $i$ be the index of any beam of the network.
First, we consider the mathematical model written in terms of the position $\mathbf{p}_i$ of the centerline and of a rotation matrix $\mathbf{R}_i$ whose columns $\{\mathbf{b}_i^\alpha\}_{\alpha=1}^3$ give the orientation of the cross sections. Both $\mathbf{p}_i$ and $\mathbf{R}_i$ depend on $x$ and $t$, with $x\in [0, \ell_i]$ where $\ell_i>0$ is the length of the beam, and both are expressed in the fixed basis $\{e_\alpha\}_{\alpha=1}^3$.
The former has values in $\mathbb{R}^3$, while the latter has values in the special orthogonal group $\mathrm{SO}(3)$.\footnote{$\mathrm{SO}(3)$ is the set of unitary real matrices of size $3$ and with a determinant equal to $1$, also called \emph{rotation} matrices.}

The columns of $\mathbf{R}_i$ may also be seen as a moving basis of $\mathbb{R}^3$, attached to the beam, and with origin $\mathbf{p}_i$; we call it \emph{body-attached basis} as opposed to the fixed basis $\{e_\alpha\}_{\alpha=1}^3$. We refer to Fig. \ref{fig:beam_netw} for visualization.

The corresponding model is called the \emph{Geometrically Exact Beam} model (GEB) and, for a freely vibrating beam, is set in $(0, \ell_i)\times(0, T)$ and reads
\begin{linenomath}
\begin{equation}
\label{eq:GEB_pres}
\partial_t \left( \begin{bmatrix}
\mathbf{R}_i & \mathbf{0}_{3}\\ \mathbf{0}_{3} & \mathbf{R}_i
\end{bmatrix} \mathbf{M}_i
\begin{bmatrix}
V_i \\ W_i
\end{bmatrix}
\right) = \partial_x \begin{bmatrix}
\phi_i \\ \psi_i \end{bmatrix} + \begin{bmatrix}
\mathbf{0}_{3, 1} \\ (\partial_x \mathbf{p}_i) \times \phi_i
\end{bmatrix},
\end{equation}
\end{linenomath}
where $V_i, W_i, \phi_i, \psi_i$ are functions of the unknowns $\mathbf{p}_i, \mathbf{R}_i$. More precisely, we introduce the linear velocity $V_i$, angular velocity $W_i$, internal forces $\Phi_i$ and internal moments $\Psi_i$ of the beam $i$, all having values in $\mathbb{R}^3$ and being expressed in the body-attached basis. They are defined by (see Subsection \ref{subsec:notation})
\begin{linenomath}
\begin{equation} \label{eq:single_beam_VWPhiPsi}
\begin{bmatrix}
V_i \\ W_i
\end{bmatrix}
= \begin{bmatrix}
\mathbf{R}_i^\intercal \partial_t \mathbf{p}_i\\ \mathrm{vec}\left( \mathbf{R}_i^\intercal \partial_t \mathbf{R}_i \right)
\end{bmatrix}, \quad 
\begin{bmatrix}
\Phi_i \\ \Psi_i
\end{bmatrix}= \mathbf{C}_i^{-1} \begin{bmatrix}
\mathbf{R}_i ^\intercal \partial_x \mathbf{p}_i  - e_1 \\ 
\mathrm{vec}\left(\mathbf{R}_i^\intercal \partial_x \mathbf{R}_i - R_i^\intercal \tfrac{\mathrm{d}}{\mathrm{d}x} R_i\right)
\end{bmatrix},
\end{equation}
\end{linenomath}
while the variables $\phi_i, \psi_i$ just correspond to $\Phi_i, \Psi_i$ when expressed in the fixed basis instead of the body-attached basis; in other words 
\begin{linenomath}
\begin{equation} \label{eq:def_smallphipsii}
\phi_i = \mathbf{R}_i \Phi_i, \quad \psi_i = \mathbf{R}_i \Psi_i.
\end{equation}
\end{linenomath}
In the above governing system and definitions, 
\begin{linenomath}
\begin{equation} \label{eq:reg_beampara}
\mathbf{M}_i, \mathbf{C}_i \in C^1([0, \ell_i]; \mathcal{S}_{++}^6), \quad R_i \in C^2([0, \ell_i]; \mathrm{SO}(3))
\end{equation}
\end{linenomath}
are the so-called \emph{mass matrix} $\mathbf{M}_i$ and \emph{flexibility matrix} $\mathbf{C}_i$ which characterize the material and geometry of the beam $i$, while $R_i$ characterizes the initial form of this beam, as it may be pre-curved and twisted before deformation (at rest). All three are given parameters of the beam.

\begin{remark}
Consider a single beam $i$ described by \eqref{eq:GEB_pres}, with homogeneous Neumann boundary conditions at each end -- i.e., both $\phi_i$ and $\psi_i$ are identically equal to zero on $\{0\}\times(0, T)$ and $\{\ell\}\times (0, T)$.
With appropriate initial conditions, rigid body motions such as defined below are solutions to the GEB model:
\begin{linenomath}
\begin{align} \label{eq:rigid_body_motion}
\mathbf{p}_i(x,t) = f(t) + \int_0^x R_i(s)e_1 ds, \qquad \mathbf{R}_i(x,t) = K(t)R_i(x)
\end{align}
\end{linenomath}
for all $(x,t) \in [0, \ell_i]\times[0, T]$, where $(f, K) \in C^2([0, T]; \mathbb{R}^3 \times \mathrm{SO}(3))$ 
are such that $\frac{\mathrm{d}}{\mathrm{d}t}f \equiv f_\circ$ and $\mathrm{vec}(K^\intercal \frac{\mathrm{d}}{\mathrm{d}t}K) \equiv k_\circ$ for some fixed $f_\circ, k_\circ \in \mathbb{R}^3$.
%$\mathbf{p}_i(x, 0) = \alpha_i + \int_0^x R_i(s)e_1 ds$, $\mathbf{R}_i(x, 0) = K_0^i R(x)$, $\partial_t\mathbf{p}_i(x, 0) = \beta_i$, and $(\mathbf{R}_iW_i)(x,0) = K_i^0 \kappa_i$.
\end{remark}

\noindent The mathematical model may also be written in terms of intrinsic variables expressed in the body-attached basis, namely, linear/angular velocities and internal forces/moments $v_i, z_i \colon [0, \ell_i]\times[0, T] \rightarrow \mathbb{R}^6$, respectively. In this case, one considers the unknown state $y_i \colon [0, \ell_i]\times[0, T] \rightarrow \mathbb{R}^{12}$ of the form
\begin{linenomath}
\begin{align} \label{eq:form_yi}
y_i = \begin{bmatrix}
v_i \\ z_i
\end{bmatrix}, \quad \text{where} \quad v_i = \begin{bmatrix}
V_i \\ W_i
\end{bmatrix}, \ z_i = \begin{bmatrix}
\Phi_i \\ \Psi_i
\end{bmatrix}.
\end{align}
\end{linenomath}
We call the corresponding model the \emph{Intrinsic Geometrically Exact Beam} model (IGEB), and it reads
\begin{linenomath}
\begin{align}
\label{eq:IGEB_pres}
\partial_t y_i + A_i(x) \partial_x y_i + \overline{B}_i(x) y_i = \overline{g}_i(x, y_i),
\end{align}
\end{linenomath}
where the coefficients $A_i,\overline{B}_i$ and the source $\overline{g}_i$ depend on $\mathbf{M}_i, \mathbf{C}_i$ and $R_i$. 
More precisely, $A_i \in C^1([0, \ell_i]; \mathbb{R}^{12 \times 12})$ is defined by (see \eqref{eq:reg_beampara})
\begin{linenomath}
\begin{align}\label{eq:def_Ai}
A_i = - \begin{bmatrix}
\mathbf{0}_6 & \mathbf{M}_i^{-1}\\
\mathbf{C}_i^{-1} & \mathbf{0}_6
\end{bmatrix},
\end{align}
\end{linenomath}
and we will see, in Subsection \ref{subsec:hyperbolic}, that the matrix $A_i(x)$ is hyperbolic for all $x \in [0, \ell_i]$ (i.e., it has real eigenvalues only, with twelve associated independent eigenvectors).

The matrix $\overline{B}_i(x)$ is indefinite and, up to the best of our knowledge, may not be assumed arbitrarily small
% -- in particular cases, the norm of this matrix can be explicitly computed and seen to be away from zero for realistic beam parameters. 
implying not only that the linearized system \eqref{eq:IGEB_pres} is not homogeneous, but also that \eqref{eq:IGEB_pres} cannot be seen as the perturbation of a system of conservation laws. The function $\overline{B}_i \in C^1([0, \ell_i];\mathbb{R}^{12 \times 12})$ which depends, just as $A_i$, on the mass and flexibility matrices, also depends on the curvature $\Upsilon_c^i \colon [0, \ell_i] \rightarrow \mathbb{R}^3$ of the beam before deformation, and is defined by
\begin{linenomath}
\begin{align*}
\overline{B}_i = \begin{bmatrix}
\mathbf{0}_6 & - \mathbf{M}^{-1}_i\mathbf{E}_i\\
\mathbf{C}_i^{-1}\mathbf{E}_i^\intercal & \mathbf{0}_6
\end{bmatrix}, \quad \text{with} \ \ \mathbf{E}_i = \begin{bmatrix}
\widehat{\Upsilon}_c^i &  \mathbf{0}_3\\
\widehat{e}_1 & \widehat{\Upsilon}_c^i
\end{bmatrix}, \quad \Upsilon_c^i = \mathrm{vec}\big(R_i^\intercal \tfrac{\mathrm{d}}{\mathrm{d}x} R_i \big).
\end{align*}
\end{linenomath}

The function $\overline{g}_i \colon [0, \ell_i]\times \mathbb{R}^{12} \rightarrow \mathbb{R}^{12}$ is defined by $\overline{g}_i(x, u) = \overline{\mathcal{G}}_i(x, u)u$ for all $x \in [0, \ell_i]$ and $u=(u_1^\intercal, u_2^\intercal, u_3^\intercal, u_4^\intercal)^\intercal \in \mathbb{R}^{12}$ with each $u_j \in \mathbb{R}^3$, where the map  $\overline{\mathcal{G}}_i$ is defined by (see Subsection \ref{subsec:notation})
\begin{linenomath}
\begin{align*}
\overline{\mathcal{G}}_i(x,u) = - 
\begin{bmatrix}
\mathbf{M}_i(x)^{-1} & \mathbf{0}_6\\
\mathbf{0}_6 & \mathbf{C}_i(x)^{-1}
\end{bmatrix}
\begin{bmatrix}
\widehat{u}_2 & \mathbf{0}_3 & \mathbf{0}_3 & \widehat{u}_3\\
\widehat{u}_1 & \widehat{u}_2 & \widehat{u}_3 & \widehat{u}_4 \\
\mathbf{0}_3 & \mathbf{0}_3 & \widehat{u}_2 & \widehat{u}_1\\
\mathbf{0}_3 & \mathbf{0}_3 & \mathbf{0}_3 & \widehat{u}_2
\end{bmatrix} 
\begin{bmatrix}
\mathbf{M}_i(x) & \mathbf{0}_6\\
\mathbf{0}_6 & \mathbf{C}_i(x)
\end{bmatrix}.
\end{align*}
\end{linenomath}
One sees that $\overline{g}_i$ is a quadratic nonlinearity (in the sense that its components are quadratic forms on $\mathbb{R}^{12}$ with respect to the second argument), and that it has the same regularity as the mass and flexibility matrices $\mathbf{M}_i, \mathbf{C}_i$ with respect to its first argument, and is $C^\infty$ with respect to its second argument. Moreover, $\overline{g}_i(x, \cdot)$ is locally Lipschitz in $\mathbb{R}^{12}$ for any $x \in [0, \ell_i]$, and $\overline{g}_i$ is locally Lipschitz in $H^1(0, \ell_i; \mathbb{R}^{12})$, but no global Lipschitz property is available.

\medskip

\noindent Finally, as mentionned in the introduction, one may see \eqref{eq:GEB_pres} and \eqref{eq:IGEB_pres} as being related by the nonlinear transformation $\mathcal{T}$ defined by (see \eqref{eq:single_beam_VWPhiPsi})
\begingroup % keep the change local
\renewcommand*{\arraystretch}{0.9}
\begin{linenomath}
\begin{equation} \label{eq:transfo}
\mathcal{T}((\mathbf{p}_i, \mathbf{R}_i)_{i \in \mathcal{I}}) = (\mathcal{T}_i(\mathbf{p}_i, \mathbf{R}_i))_{i \in \mathcal{I}}, \quad \text{where} \quad
\mathcal{T}_i (\mathbf{p}_i, \mathbf{R}_i) = 
\begin{bmatrix} V_i\\ W_i\\ \Phi_i\\ \Psi_i\end{bmatrix}.
\end{equation}
\end{linenomath}
\endgroup

\subsection{Dynamics of the network of beams}
\label{subsec:network_systems}

\begin{figure}
  \begin{subfigure}{0.2\textwidth}
  \centering
    \includegraphics[height=3.25cm]{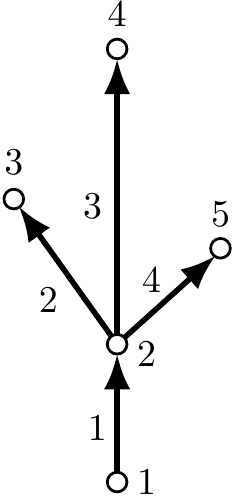}
    \caption{Star-shaped}
  \end{subfigure}%
  \begin{subfigure}{0.24\textwidth}
    \centering
    \includegraphics[height=3.25cm]{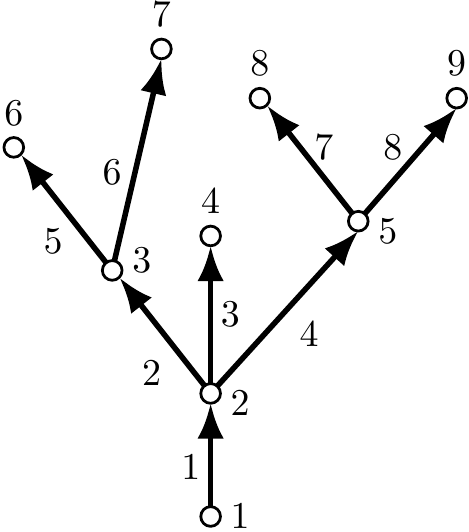}
    \caption{Tree-shaped} 
  \end{subfigure}%
  \begin{subfigure}{0.28\textwidth}
    \centering
    \includegraphics[height=3.25cm]{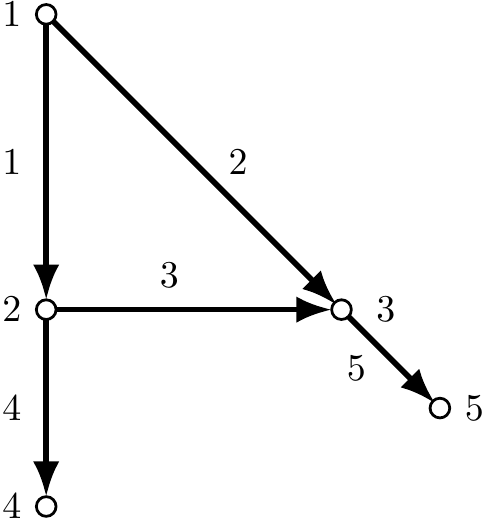}
    \caption{A-shaped} 
    \label{subfig:AshapedNetwork}
  \end{subfigure}%
    \hspace*{\fill}   % maximizeseparation between the subfigures
  \begin{subfigure}{0.26\textwidth}
    \centering
    \includegraphics[width=2.75cm]{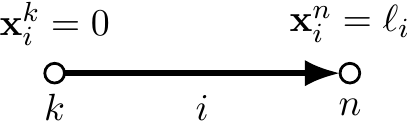}
\caption{Orientation of an edge $i$ starting and ending at the nodes $k$ and $n$, respectively.}
\label{fig:xin}
  \end{subfigure}

\caption{Some oriented graphs representing beam networks, and orientation of the edges.}
\label{fig:exple_networks}
\end{figure}

Let us now give the systems describing the entire beam network.

\subsubsection{Network notation}
To represent a collection of $N$ beams attached in a certain manner to each other at their tips, we use an oriented graph containing $N$ edges. Any edge $i$ is identified with the interval $[0, \ell_i]$, which is the spatial domain for the beam model in question (GEB or IGEB). Hence, just as for the beams, the \emph{edges} are indexed by $i \in \mathcal{I} = \{1, \ldots, N\}$, while the \textit{nodes} are indexed by $n \in \mathcal{N} = \{1, \ldots, \#\mathcal{N}\}$, where $\#$ denotes the set cardinality. The set of nodes is partitioned as $\mathcal{N} = \mathcal{N}_S \cup \mathcal{N}_M$, where $\mathcal{N}_S$ is the set of indexes of \emph{simple nodes}, while $\mathcal{N}_M$ is the set of indexes \emph{multiple nodes}. 

The former set is in addition partitioned as $\mathcal{N}_S = \mathcal{N}_S^D \cup \mathcal{N}_S^N$, where $\mathcal{N}_S^D$ contains the simple nodes with prescribed \emph{Dirichlet} boundary conditions (i.e., the centerline's position and the cross section's orientation in the case of the GEB model, or the velocities in the case of the IGEB model, are prescribed), while $\mathcal{N}_S^N$ contains the simple nodes with prescribed \emph{Neumann} boundary conditions (i.e., the internal forces and moments are prescribed).

\medskip

\noindent For any $n\in\mathcal{N}$, we denote by $\mathcal{I}^n$ the set of indexes of edges incident to the node $n$, by $k_n = \# \mathcal{I}^n $ the \emph{degree} of the node $n$, and by $i^n$ the index\footnote{Defining $i^n$ as the \emph{smallest} element of $\mathcal{I}^n$, and not the \emph{largest} for example, is an arbitrary choice and is of no influence here.}
\begin{linenomath}
\begin{align} \label{eq:def_in}
i^n = \min_{i\in \mathcal{I}^n} i.
\end{align}
\end{linenomath}
Note that in the case of a simple node, $\mathcal{I}^n = \{i^n\}$.

The orientation of each beam is given by the variables $\mathbf{x}_i^n$ and $\tau_i^n$ defined as follows.
For any $i \in \mathcal{I}^n$, we denote by $\mathbf{x}_i^n$ the end of the interval $[0, \ell_i]$ which corresponds to the node $n$, while $\tau_i^n$ is the outward pointing normal at $\mathbf{x}_i^n$: 
\begin{linenomath}
\begin{align*}
\tau_i^n = 
\left\{ 
\begin{aligned}
-1 \qquad & \text{if } \mathbf{x}_i^n = 0,\\
+1 \qquad &\text{if } \mathbf{x}_i^n = \ell_i.
\end{aligned}
\right.
\end{align*}
\end{linenomath}
As described in Fig. \ref{fig:exple_networks}, each edge $i$ is represented by an arrow and each node $n$ by a circle. The arrowhead is at the ending point $x=\ell_i$; see Fig. \ref{fig:xin}.

\subsubsection{The network model}

Let $T > 0$.
If all beams are described by the GEB model \eqref{eq:GEB_pres}, then the overall network is described by System \eqref{eq:GEB_netw} below, which gives the dynamics of the unknown state $(\mathbf{p}_i, \mathbf{R}_i)_{i \in \mathcal{I}}$:
\begin{linenomath}
\begin{subnumcases}{\label{eq:GEB_netw}}
\nonumber
\partial_t \left( 
\left[\begin{smallmatrix}
\mathbf{R}_i & \mathbf{0}_3\\
\mathbf{0}_3 & \mathbf{R}_i
\end{smallmatrix}\right]
\mathbf{M}_i
\left[\begin{smallmatrix}
V_i \\ W_i
\end{smallmatrix}\right]
\right) & \\
\label{eq:GEB_gov}
\hspace{1cm}= \partial_x \left[\begin{smallmatrix}
\phi_i \\ \psi_i \end{smallmatrix} \right] + \left[\begin{smallmatrix}
\mathbf{0}_{3, 1} \\ (\partial_x \mathbf{p}_i) \times \phi_i
\end{smallmatrix} \right] 
&$\text{in } (0, \ell_i)\times(0, T), \, i \in \mathcal{I}$\\
\label{eq:GEB_continuity_pi}
\mathbf{p}_i(\mathbf{x}_i^n, t)  = \mathbf{p}_{i^n}(\mathbf{x}_{i^n}^n, t) &$t \in (0, T), \, i \in \mathcal{I}^n, \, n \in \mathcal{N}_M$\\
\label{eq:GEB_rigid_angles}
(\mathbf{R}_i R_{i}^\intercal)(\mathbf{x}_i^n, t) = (\mathbf{R}_{i^n} R_{i^n}^\intercal)(\mathbf{x}_{i^n}^n, t) &$t \in (0, T), \, i \in \mathcal{I}^n, \, n \in \mathcal{N}_M$\\
\label{eq:GEB_Kirchhoff}
{\textstyle \sum_{i\in\mathcal{I}^n}} \tau_i^n \left[ \begin{smallmatrix} 
\phi_i \\ \psi_i
\end{smallmatrix} \right] (\mathbf{x}_i^n, t) = f_n (t)
&$t \in (0, T), \, n \in \mathcal{N}_M$\\
\label{eq:GEB_condNSz}
\tau_{i^n}^n \left[ \begin{smallmatrix} 
\phi_{i^n} \\ \psi_{i^n}
\end{smallmatrix} \right] (\mathbf{x}_{i^n}^n, t) = f_n (t) &$t \in (0, T), \, n \in \mathcal{N}_S^N$\\
\label{eq:GEB_condNSv_p_R}
(\mathbf{p}_{i^n}, \mathbf{R}_{i^n})(\mathbf{x}_{i^n}^n, t) = (f_n^\mathbf{p}, f_n^\mathbf{R})(t) &$t \in (0, T), \, n \in \mathcal{N}_S^D$\\
\label{eq:GEB_IC_0ord}
(\mathbf{p}_i, \mathbf{R}_i)(x, 0) = (\mathbf{p}_i^0, \mathbf{R}_i^0)(x) &$x \in (0, \ell_i), \, i \in \mathcal{I}$\\
\label{eq:GEB_IC_1ord}
(\partial_t \mathbf{p}_i, \mathbf{R}_i W_i)(x, 0) = (\mathbf{p}_i^1, w_i^0)(x) &$x \in (0, \ell_i), \, i \in \mathcal{I}$,
\end{subnumcases}
\end{linenomath}
where we recall that $V_i, W_i, \phi_i, \psi_i$ are defined in \eqref{eq:single_beam_VWPhiPsi}-\eqref{eq:def_smallphipsii}.
In this system, \eqref{eq:GEB_IC_0ord}-\eqref{eq:GEB_IC_1ord} describe the initial conditions, with data
\begin{linenomath}
\begin{align} \label{eq:reg_Idata_GEB}
(\mathbf{p}_i^0, \mathbf{R}_i^0) \in C^2([0, \ell_i]; \mathbb{R}^3 \times \mathrm{SO}(3)), \quad \mathbf{p}_i^1, w_i^0 \in C^1([0, \ell_i]; \mathbb{R}^3), \quad i \in\mathcal{I}.
\end{align}
\end{linenomath}
Then, \eqref{eq:GEB_continuity_pi}-\eqref{eq:GEB_rigid_angles}-\eqref{eq:GEB_Kirchhoff}
are the so-called \emph{transmission} (or \emph{interface}) conditions for multiple nodes, while the conditions \eqref{eq:GEB_condNSz}-\eqref{eq:GEB_condNSv_p_R} are enforced at simple nodes. The nodal data is
\begin{linenomath}
\begin{align}
\label{eq:reg_Ndata_GEB_N}
f_n \in C^1([0, T]; \mathbb{R}^{6}), \quad &n \in \mathcal{N}_M \cup \mathcal{N}_S^N\\
\label{eq:reg_Ndata_GEB_D}
(f_n^\mathbf{p}, f_n^\mathbf{R}) \in C^2([0, T]; \mathbb{R}^3\times \mathrm{SO}(3)), \quad &n \in\mathcal{N}_S^D.
\end{align}
\end{linenomath}

\medskip

\noindent On the other hand, if all beams are described by the IGEB model \eqref{eq:IGEB_pres}, then for the overall network, the unknown state $(y_i)_{i \in \mathcal{I}}$ is described by System \eqref{eq:syst_physical}, which reads
\begin{linenomath}
\begin{subnumcases}{\label{eq:syst_physical}}
\label{eq:IGEB_gov}
\partial_t y_i + A_i \partial_x y_i + \overline{B}_i y_i = \overline{g}_i(\cdot,y_i) &$\text{in } (0, \ell_i)\times(0, T), \, i \in \mathcal{I}$\\
\label{eq:IGEB_cont_velo}
(\overline{R}_i v_i)(\mathbf{x}_i^n, t) = (\overline{R}_{i^n} v_{i^n})(\mathbf{x}_{i^n}^n, t) &$t \in (0, T), \, i \in \mathcal{I}^n, \, n \in \mathcal{N}_M$\\
\label{eq:IGEB_Kirchhoff}
\sum_{i\in\mathcal{I}^n} \tau_i^n (\overline{R}_i z_i)(\mathbf{x}_i^n, t) = q_n(t) &$t \in (0, T), \, n \in \mathcal{N}_M$\\
\label{eq:IGEB_condNSz}
\tau_{i^n}^n z_{i^n} (\mathbf{x}_{i^n}^n, t) = q_n(t) &$t \in (0, T), \, n \in \mathcal{N}_S^N$\\
\label{eq:IGEB_condNSv}
v_{i^n}(\mathbf{x}_{i^n}^n, t) = q_n(t) &$t \in (0, T), \, n \in \mathcal{N}_S^D$\\
\label{eq:IGEB_ini_cond}
y_i(x, 0) = y_i^0(x) &$x \in (0, \ell_i), \, i \in \mathcal{I}$,
\end{subnumcases}
\end{linenomath}
$v_i,z_i$ representing the first and last six components of $y_i$, respectively (see \eqref{eq:form_yi}), and where $\overline{R}_i \in C^2([0, \ell_i]; \mathbb{R}^{6 \times 6})$ is defined by $\overline{R}_i = \mathrm{diag}(R_i, R_i)$ (see \eqref{eq:reg_beampara}).
Here, \eqref{eq:IGEB_ini_cond} gives the initial conditions, with data 
\begin{linenomath}
\begin{align} \label{eq:reg_Idata_IGEB}
y_i^0 \in C^1([0, \ell_i]; \mathbb{R}^{12}), \quad i \in \mathcal{I},
\end{align}
\end{linenomath}
the transmission conditions are \eqref{eq:IGEB_cont_velo}-\eqref{eq:IGEB_Kirchhoff}, while the conditions \eqref{eq:IGEB_condNSz}-\eqref{eq:IGEB_condNSv} are imposed at simple nodes, with data
\begin{linenomath}
\begin{align}\label{eq:reg_Ndata_IGEB}
q_n \in C^1([0, T]; \mathbb{R}^{6}), \quad n \in \mathcal{N}.
\end{align}
\end{linenomath}

\subsubsection{Origin of the nodal conditions}

As the form of transmission conditions is an essential aspect in the proof of nodal profile controllability of hyperbolic systems on networks, let us now explain the origin of these conditions for System \eqref{eq:GEB_netw} and especially those of System \eqref{eq:syst_physical}. See also \cite{R2020} for a more detailed presentation, and for the meaning of the states and coefficients of \eqref{eq:GEB_netw} and \eqref{eq:syst_physical}.

\medskip

\noindent Let $n$ be the index of some multiple node. 
In this work, we assume that, at all times, the beams incident with this node remain attached to each other. In other words, as imposed by \eqref{eq:GEB_continuity_pi}, the position of their centerlines must coincide. 
Moreover, we work under the \emph{rigid joint} assumption, namely, at any node, there is no relative motion between the incident beams. As the orientation of the cross sections before deformation is specified by the (given) function $R_i$, the rigid joint assumption is enforced by the condition \eqref{eq:GEB_rigid_angles} which states that the change of orientation $\mathbf{R}_iR_i^\intercal$ (from the undeformed state of the beam network to its state at time $t$) is the same for all incident beams. See also \cite[Subsection 2.4]{strohm_dissert}.

For the IGEB model, the condition corresponding to the continuity of the centerline's position and of the change of the cross section's orientation, is the \emph{continuity} of velocities \eqref{eq:IGEB_cont_velo}. Indeed, one may differentiate \eqref{eq:GEB_continuity_pi} and \eqref{eq:GEB_rigid_angles} with respect to time, and then left-multiply each of the obtained equations by $(R_j\mathbf{R}_j^\intercal)(\mathbf{x}_j^n, t)$ for the corresponding beam index $j$ (thereby using the rigid joint assumption), to obtain
\begin{linenomath}
\begin{align*}
(R_i\mathbf{R}_i^\intercal \partial_t \mathbf{p}_i)(\mathbf{x}_i^n, t) &= (R_{i^n}\mathbf{R}_{i^n}^\intercal \partial_t \mathbf{p}_{i^n})(\mathbf{x}_{i^n}^n, t), \\
(R_i\mathbf{R}_i^\intercal \partial_t \mathbf{R}_i R_i^\intercal)(\mathbf{x}_i^n, t) &= (R_{i^n}\mathbf{R}_{i^n}^\intercal \partial_t \mathbf{R}_{i^n} R_{i^n}^\intercal)(\mathbf{x}_{i^n}^n, t),
\end{align*}
\end{linenomath}
respectively. The above equations turn out to equate to \eqref{eq:IGEB_cont_velo}, by the definition of $V_i$ and $W_i$ (see \eqref{eq:single_beam_VWPhiPsi}), and by using that the invariance of the cross product in $\mathbb{R}^3$ under rotation provides the identity $R_i \widehat{W}_iR_i^\intercal = \widehat{R_iW_i}$.

\medskip

\noindent Furthermore, at this multiple node $n$, we require the internal forces $\phi_i$ and moments $\psi_i$ exerted by incident beams $i\in\mathcal{I}_n$ to be balanced with the external load $f_n$ applied at this node, which reads as \eqref{eq:GEB_Kirchhoff}, and is also called the \emph{Kirchhoff} condition.

The corresponding Kirchhoff condition \eqref{eq:IGEB_Kirchhoff} for the IGEB model is then obtained by left-multiplying each term in the right-hand side of \eqref{eq:GEB_Kirchhoff} by $(R_i \mathbf{R}_i^\intercal)(\mathbf{x}_i^n, t)$ for the corresponding index $i$ (once again using the rigid joint assumption), left-multiplying $f_n$ by $(R_i \mathbf{R}_i^\intercal)(\mathbf{x}_i^n, t)$ for some $i \in \mathcal{I}_n$ (for instance as $i^n$), and recalling the relationship between $\phi_i, \psi_i$ and $\Phi_i, \Psi_i$ (see \eqref{eq:def_smallphipsii}).

\medskip

\noindent Similar considerations hold for simple nodes. Here, either $n \in \mathcal{N}_S^N$ and an external load $f_n$ is applied at this node, yielding the condition \eqref{eq:GEB_condNSz}, or $n \in \mathcal{N}_S^D$ and the centerline's position and cross section's orientation are prescribed as $f_n^\mathbf{p}$ and $f_n^\mathbf{R}$, respectively, for the beam $i^n$ incident with this node, yielding the condition \eqref{eq:GEB_condNSv_p_R}.

For the IGEB model, this translates to \eqref{eq:IGEB_condNSz} and \eqref{eq:IGEB_condNSv}, respectively, when one left-multiplies \eqref{eq:GEB_condNSz} and \eqref{eq:GEB_condNSv_p_R} by $(R_{i^n}\mathbf{R}_{i^n}^\intercal)(\mathbf{x}_{i^n}^n, t)$ and $\mathbf{R}_{i^n}^\intercal(\mathbf{x}_{i^n}^n, t)$, respectively.

\subsubsection{Relationship between the data of both systems}

As mentioned earlier, the unknowns of the GEB and IGEB models are related by the transformation $\mathcal{T}$, defined in \eqref{eq:transfo}.
Thus, the initial data of both models are related as follows: for given $\mathbf{p}_i^0, \mathbf{R}_i^0, \mathbf{p}_i^1$ and $w_i^0$, one has 
\begin{linenomath}
\begin{align} \label{eq:rel_inidata}
y_i^0 = \begin{bmatrix}
v_i^0 \\ z_i^0
\end{bmatrix}, \quad
v_i^0 = \begin{bmatrix}
(\mathbf{R}_i^0)^{\intercal} \mathbf{p}_i^1 \\
(\mathbf{R}_i^0 )^{\intercal} w_i^0
\end{bmatrix}
, \quad
z_i^0 = \mathbf{C}_i^{-1} 
\begin{bmatrix}
(\mathbf{R}_i^0)^{\intercal} \frac{\mathrm{d}}{\mathrm{d}x} \mathbf{p}_i^0 - e_1\\
\mathrm{vec}\left( (\mathbf{R}_i^0)^{\intercal} \frac{\mathrm{d}}{\mathrm{d}x}\mathbf{R}_i^0 - R_i^{\intercal}\frac{\mathrm{d}}{\mathrm{d}x} R_i \right)
\end{bmatrix}.
\end{align}
\end{linenomath}
Similarly, the nodal conditions of \eqref{eq:GEB_netw} and \eqref{eq:syst_physical} are connected via $\mathcal{T}$, and with the help of the above considerations on the nodal conditions, one can observe the following relationships between the nodal data of both systems.
For any $n \in \mathcal{N}_S^D$, for given $(f_n^\mathbf{p}, f_n^\mathbf{R})$ of regularity \eqref{eq:reg_Ndata_GEB_D}, one has
\begin{linenomath}
\begin{align} \label{eq:def_qn_D}
q_n = \begin{bmatrix}
(f_n^\mathbf{R})^\intercal \frac{\mathrm{d}}{\mathrm{d}t}f_n^\mathbf{p}\\
(f_n^\mathbf{R})^\intercal \frac{\mathrm{d}}{\mathrm{d}t}f_n^\mathbf{R}
\end{bmatrix},
\end{align}
\end{linenomath}
while for any $n \in \mathcal{N}_M \cup \mathcal{N}_S^N$,
\begin{linenomath}
\begin{align} \label{eq:def_fn}
f_n  = 
\left\{
\begin{aligned}
&\mathrm{diag}\left((\mathbf{R}_{i^n} R_{i^n}^\intercal )(\mathbf{x}_{i^n}^n, \cdot), (\mathbf{R}_{i^n} R_{i^n}^\intercal) (\mathbf{x}_{i^n}^n, \cdot)\right) q_n  &&n \in \mathcal{N}_M\\
&\mathrm{diag}\big(\mathbf{R}_{i^n}(\mathbf{x}_{i^n}^n, \cdot), \mathbf{R}_{i^n}(\mathbf{x}_{i^n}^n, \cdot)\big)  q_n  &&n \in \mathcal{N}_S^N.
\end{aligned}
\right.
\end{align}
\end{linenomath}

\subsection{Main results}
\label{subsec:main_results}

We may now present our main results, which are divided in two parts: one is concerned with the well-posedness and controllability of the IGEB network, and the other with showing that the transformation from the GEB to the IGEB network is invertible, by means of which one can deduce corresponding results for the former model. 

\subsubsection{Study of the IGEB model}

Let us define compatibility conditions for System \eqref{eq:syst_physical}. As for the unknown, we write the initial data $(y_i^0)_{i \in \mathcal{I}}$ as
\begin{linenomath}
\begin{equation*}
y_i^0 = \begin{bmatrix}
v_i^0 \\ z_i^0
\end{bmatrix} , \qquad \text{with } v_i^0, z_i^0 \colon [0, \ell_i] \rightarrow \mathbb{R}^6.
\end{equation*}
\end{linenomath}

\begin{definition}
We say that the initial data $y_i^0 \in C^1([0, \ell_i]; \mathbb{R}^{12})$, for all $i\in \mathcal{I}$, and boundary data $q_n\in C^0([0, T]; \mathbb{R}^6)$, for all $n \in \mathcal{N}$, fulfill the first-order compatibility conditions of \eqref{eq:syst_physical} if
\begin{linenomath}
\begin{equation} \label{eq:compat_0} 
\begin{aligned}
&(\overline{R}_i v_i^0)(\mathbf{x}_i^n) = (\overline{R}_j v_j^0)(\mathbf{x}_j^n) \qquad && i,j \in \mathcal{I}^n, \, n \in \mathcal{N}_M\\
&{\textstyle \sum_{i\in\mathcal{I}^n}} \tau_i^n (\overline{R}_i z_i^0)(\mathbf{x}_i^n) = q_n(0) && n \in \mathcal{N}_M\\
& \tau_{i^n}^n z_{i^n}^0(\mathbf{x}_{i^n}^n) = q_n(0) && n \in \mathcal{N}_S^N\\
& v_{i^n}^0 (\mathbf{x}_{i^n}^n) = q_n(0) && n \in \mathcal{N}_S^D,
\end{aligned}
\end{equation} 
\end{linenomath}
holds and  $y_i^1 \in C^0([0, \ell_i]; \mathbb{R}^{12})$, for all $i\in \mathcal{I}$, defined by 
\begin{linenomath}
\begin{equation*}
y_i^1 = -  A_i \frac{\mathrm{d}y_i^0}{\mathrm{d}x} - \overline{B}_i y_i^0 + \overline{g}_i(\cdot, y_i^0) = \begin{bmatrix}
v_i^1 \\z_i^1
\end{bmatrix},
\end{equation*}
\end{linenomath}
also fulfills \eqref{eq:compat_0}, where $v_i^0, z_i^0$ are replaced by $v_i^1, z_i^1$ respectively. 
\end{definition}

In order to ensure a certain regularity of the eigenvalues and eigenvectors of $A_i$, we will later on make the following assumption.

\begin{assumption} \label{as:mass_flex}
For all $i \in \mathcal{I}$, we suppose that
\begin{enumerate}
\item \label{eq:assump1_1} $\mathbf{C}_i, \mathbf{M}_i \in C^2([0, \ell_i]; \mathcal{S}_{++}^6)$;
\item \label{eq:assump1_2} the function $\Theta_i \in C^2([0, \ell_i]; \mathcal{S}_{++}^6)$ defined by $\Theta_i = (\mathbf{C}_i^{\sfrac{1}{2}} \mathbf{M}_i\mathbf{C}_i^{\sfrac{1}{2}})^{-1}$, is such that there exists $U_i, D_i \in C^2([0, \ell_i]; \mathbb{R}^{6 \times 6})$ for which
\begin{linenomath}
\begin{align*}
\Theta_i = U_i^\intercal D_i^2 U_i, \quad \text{in }[0, \ell_i],
\end{align*}
\end{linenomath}
where $D_i(x)$ is a positive definite diagonal matrix containing the square roots of the eigenvalues of $\Theta_i(x)$ as diagonal entries, while $U_i(x)$ is unitary.
\end{enumerate}
\end{assumption}

One may note that, in Assumption \ref{as:mass_flex}, if \ref{eq:assump1_1} holds, then \ref{eq:assump1_2} is readily verified if $\mathbf{M}_i, \mathbf{C}_i$ have values in the set of diagonal matrices, or if the eigenvalues of $\Theta_i(x)$ are distinct for all $x \in [0, \ell_i]$ (one may adapt \cite[Th. 2, Sec. 11.1]{evans2}). Clearly, \ref{eq:assump1_2} is also satisfied if $\mathbf{M}_i, \mathbf{C}_i$ are constant, entailing that the material and geometrical properties of the beam do not vary along its centerline.

\medskip

\noindent Our first task is to obtain the existence and uniqueness of semi-global in time solutions to \eqref{eq:syst_physical} for any network. Henceforth, in the norms' subscripts, when there is no ambiguity, we use the abbreviations $C_x^1 = C^1([0, \ell_i]; \mathbb{R}^d)$, $C_t^1 = C^1(I; \mathbb{R}^d)$ and $C_{x,t}^1 = C^1([0, \ell_i]\times I; \mathbb{R}^d)$ for the appropriate time interval $I$ and dimension $d \in \{1, 2, \ldots\}$.

\begin{theorem} \label{th:existence}
Consider a general network, suppose that $R_i$ has the regularity \eqref{eq:reg_beampara} and that Assumption \ref{as:mass_flex} is fulfilled.
Then, for any $T>0$, there exists $\varepsilon_0>0$ such that for all $\varepsilon \in (0, \varepsilon_0)$ and for some $\delta>0$, and all initial and boundary data $y_i^0, q_n$ of regularity \eqref{eq:reg_Idata_IGEB}-\eqref{eq:reg_Ndata_IGEB}, and satisfying $\|y_i^0\|_{C_x^1} +\|q_n\|_{C_t^1} \leq \delta$ and the first-order compatibility conditions of \eqref{eq:syst_physical}, there exists a unique solution $(y_i)_{i \in \mathcal{I}} \in \prod_{i=1}^N C^1([0, \ell_i] \times [0, T]; \mathbb{R}^{12})$ to \eqref{eq:syst_physical}, with $\|y_i\|_{C_{x,t}^1}\leq \varepsilon$.
\end{theorem}

The proof of Theorem \ref{th:existence}, given in Section \ref{sec:exist}, consists in rewriting \eqref{eq:syst_physical} as a single hyperbolic system and applying general well-posedness results \cite{li2010controllability, wang2006exact}. To do so, one has to write \eqref{eq:syst_physical} in Riemann invariants, the new unknown state being denoted $(r_i)_{i \in \mathcal{I}}$, and verify that the nodal conditions fulfill the following rule: at any node, the components of $r_i$ corresponding to characteristics \emph{entering} the domain $[0, \ell_i]\times [0, +\infty)$ at this node is expressed explicitly as a function of the components of $r_i$ corresponding to characteristics \emph{leaving} the domain $[0, \ell_i]\times [0, +\infty)$ at this node (more detail is given in Subsection \ref{subsec:out_in_info}).

\begin{remark} 
Assuming that $R_i \in C^2([0, \ell_i]; \mathrm{SO}(3))$ guaranties that $\overline{B}_i \in C^1([0, \ell_i]; \mathbb{R}^{12 \times 12})$. On the other hand, in Assumption \ref{as:mass_flex}, the extra regularity for $\mathbf{M}_i, \mathbf{C}_i$ ($C^2$, instead of $C^1$ as in \eqref{eq:reg_beampara}) permits us to ensure that the coefficients of the System \eqref{eq:syst_physical} written in Riemann invariants, in particular $B_i$ (see Subsection \ref{subsec:change_var}), are sufficiently regular.
\end{remark}

We now consider a problem of local exact boundary controllability of nodal profiles, for the specific case of the A-shaped network illustrated in Fig. \ref{subfig:AshapedNetwork}, consisting of five nodes and five edges and having one cycle. More precisely, we consider the network defined by
\begin{linenomath}
\begin{align} \label{eq:A_netw}
\begin{aligned}
&\mathcal{N}_S = \mathcal{N}_S^N = \{4, 5\}, \ \mathcal{N}_M = \{1, 2, 3\}, \ \mathcal{I} = \{1, \ldots, 5\}\\
&\mathbf{x}_1^1 = 0, \ \ \mathbf{x}_2^1 = 0, \ \ \mathbf{x}_3^2 = 0, \ \ \mathbf{x}_4^2 = 0, \ \ \mathbf{x}_5^3 = 0\\
&\mathbf{x}_1^2 = \ell_1, \ \mathbf{x}_2^3 = \ell_2, \ \mathbf{x}_3^3 = \ell_3, \ \mathbf{x}_4^4 = \ell_4, \ \mathbf{x}_5^5 = \ell_5.
\end{aligned}
\end{align}
\end{linenomath}
Let us first introduce some notation concerning the eigenvalues $\{\lambda_i^k (x)\}_{k=1}^{12}$ of $A_i(x)$ for $i \in \mathcal{I}$ and $x \in [0, \ell_i]$, which, as we will see in in Subsection \ref{subsec:hyperbolic}, are such that $\{\lambda_i^k\}_{k=1}^{12} \subset C^2([0, \ell_i])$ under Assumption \ref{as:mass_flex}, and
\begin{linenomath}
\begin{align} \label{eq:sign_eigval}
\lambda_i^k(x) <0 \ \text{ if } \ k\leq 6, \qquad \lambda_i^k(x) >0 \ \text{ if } \ k\geq 7.
\end{align}
\end{linenomath}
Also under Assumption \ref{as:mass_flex}, and for any $i\in\mathcal{I}$, we define $\Lambda_i \in C^0([0, \ell_i]; (0, +\infty))$ and $T_i>0$ by
\begin{linenomath}
\begin{align} \label{eq:def_Lambdai_Ti}
    \Lambda_i(x) = \left( \min_{k \in \{1, \ldots, 6\}} \left| \lambda_i^k(x) \right| \right)^{-1} \quad \text{and} \quad T_i = \int_0^{\ell_i} \Lambda_i(x) dx;
\end{align}
\end{linenomath}
note that the minimum ranges over the \textit{negative} eigenvalues of $A_i(x)$. 
The latter, $T_i$, corresponds to the transmission (or travelling) time from one end of the beam $i$ to its other end (see Section \ref{sec:controllability}).

\begin{theorem} \label{th:controllability}
Consider the A-shaped network defined by \eqref{eq:A_netw}.
Suppose that $R_i$ has the regularity \eqref{eq:reg_beampara} and that Assumption \ref{as:mass_flex} is fulfilled.
Let $\overline{T}>0$ be defined by (see \eqref{eq:def_Lambdai_Ti})
\begin{linenomath}
\begin{align} \label{eq:minT}
\overline{T} = \max \left\{T_1, T_2 \right\} + \max \left\{T_4, T_5 \right\}. 
\end{align}
\end{linenomath}
Then, for any $T> T^*>\overline{T}$, there exists $\varepsilon_0>0$ such that for all $\varepsilon \in (0, \varepsilon_0)$, for some $\delta, \gamma>0$, and
\begin{enumerate}[label=(\roman*)]
\item for all initial data $(y_i^0)_{i \in \mathcal{I}}$ and boundary data $(q_n)_{n \in \{1, 2, 3\}}$ of regularity \eqref{eq:reg_Idata_IGEB}-\eqref{eq:reg_Ndata_IGEB}, satisfying $\|y_i^0\|_{C_x^1} + \|q_n\|_{C_t^1} \leq \delta$ and the first-order compatibility conditions of \eqref{eq:syst_physical}, and
\item for all nodal profiles $\overline{y}_1, \overline{y}_2 \in C^1([T^*, T]; \mathbb{R}^{12})$, satisfying $\|\overline{y}_i\|_{C_t^1} \leq \gamma$ and the transmission conditions \eqref{eq:IGEB_cont_velo}-\eqref{eq:IGEB_Kirchhoff} at the node $n=1$,
\end{enumerate}
there exist controls $q_4, q_5 \in C^1([0, T]; \mathbb{R}^6)$ with $\|q_i\|_{C_t^1}\leq \varepsilon$, such that \eqref{eq:syst_physical} admits a unique solution $(y_i)_{i \in \mathcal{I}} \in \prod_{i=1}^N C^1([0, \ell_i] \times [0, T]; \mathbb{R}^{12})$, which fulfills $\|y_i\|_{C_x^1} \leq \varepsilon$ and
\begin{linenomath}
\begin{align}\label{eq:aim}
y_i(0, t) = \overline{y}_i(t) \quad \text{for all }i \in \{1, 2\}, \, t \in [T^*, T].
\end{align}
\end{linenomath}
\end{theorem}

As mentionned in Section \ref{sec:intro}, the proof of Theorem \ref{th:controllability}, given in Section \ref{sec:controllability}, relies upon the existence and uniqueness theory of \emph{semi-global} classical solutions to the network problem (here, Theorem \ref{th:existence}), the form of the transmission condition of the network, and on a \emph{constructive method}. The idea of the proof is to construct a solution $(y_i)_{i\in \mathcal{I}}$ to \eqref{eq:syst_physical}, such that it satisfies the initial condition, the nodal conditions, and the given nodal profiles. Substituting this solution into the nodal conditions at the nodes $n \in \{4, 5\}$, one then obtains the desired controls $q_4, q_5$. \textcolor{black}{Our proof follows the lines of \cite{Zhuang2018}, where the authors develop a methodology for proving the nodal profile controllability for A-shaped networks of canals governed by the Saint-Venant equations.}

\begin{remark} \label{rem:controllability_thm}
A few remarks are in order.
\begin{enumerate}

\item \label{subrem:global}
{\color{black}
The smallness of the initial and nodal data and of the nodal profiles in (i) and (ii) is used to ensure the well-posedness of the mixed initial-boundary value problem for beams described by the IGEB model. This limitation leads to the local nature of the controllability result: in a sufficiently small $C^1$-neighborhood of the zero-steady state, we can construct continuously differentiable controls, which then generate a piecewise continuously differentiable solution on the whole network.  
Furthermore, the study of other equilibrium solutions for the IGEB network is relevant to achieve further objectives. For instance, in the spirit of \cite{gugatLeugering2003}, supposing that the set of equilibria is connected, one might look to use the result of local exact controllability of nodal profiles as a basis to then prove more global results.
}

\item {\color{black}
For the system linearized around the zero-steady state, a global nodal profile controllability result is also achieved, though without any limitation on the size of the of the data and nodal profiles, as it rests on an existence and uniqueness result that does not impose such limitations. Moreover, the `optimal' estimate for the controllability time $T^*$ remains that of in Theorem \ref{th:controllability}, given in terms of the transmission times \eqref{eq:def_Lambdai_Ti}.
}

\item 
The \emph{controllability time} $T^*$ from which one can prescribe nodal profiles, has to be large enough, depending on the lengths of the beams and the eigenvalues of $(A_i)_{i \in \mathcal{I}}$ (and thus, it depends on the geometrical and material properties of the beam). As we will see in Section \ref{sec:controllability}, $\overline{T}$ is the transmission time from the controlled nodes to the charged node. One may note that $T_i \leq \frac{\ell_i}{|\lambda_i^*|}$, where the constant $\lambda_i^*<0$ denotes the maximum over $x$ of the largest \emph{negative} eigenvalue of $A_i(x)$.

\item 
One will observe in the proof of Theorem \ref{th:controllability} that the controls $q_4, q_5$ are not unique due to the use of interpolation and arbitrary nodal conditions throughout the proof. 

\item \label{subrem:sidewise}
In the proof of Theorem \ref{th:controllability}, to construct the solution $(y_i)_{i\in\mathcal{I}}$, one is led to solve a series of forward and sidewise problems for \eqref{eq:IGEB_gov} for the different beams $i \in \mathcal{I}$ of the network. Solving a sidewise problem  for \eqref{eq:IGEB_gov} entails changing the role of $x$ and $t$, considering a governing system of the form 
\begin{linenomath}
\begin{align*}
\partial_x y_i + A_i^{-1}\partial_t y_i + A_i^{-1}\overline{B}_i y_i = A_i^{-1}\overline{g}_i(\cdot, y_i)
\end{align*}
\end{linenomath}
and providing ``boundary conditions'' at $t=0$ and $t = T$, and ``initial conditions'' at $x=0$ (rightward problem) or $x = \ell_i$ (leftward problem). It is consequently important here that $A_i$ does not have any zero eigenvalue.

\end{enumerate}
\end{remark}

\subsubsection{Study of the GEB model}

In order to translate Theorems \ref{th:existence} and \ref{th:controllability} in terms of the GEB model, we prove the Theorem \ref{thm:solGEB} below, which yields the existence of a unique classical solution to \eqref{eq:GEB_netw}, provided that a unique classical solution exists for \eqref{eq:syst_physical} and that the data of both models fulfill some compatibility conditions. 

Let us first introduce the compatibility conditions on the initial and boundary data of the GEB network \eqref{eq:GEB_netw}, that will be of use in the theorem and corollaries that follow
\begin{linenomath}
\begin{subequations}\label{eq:compat_GEB_-1_GEBtransmi}
\begin{align} \label{eq:compat_GEB_-1}
&(f_n^\mathbf{p}, f_n^\mathbf{R})(0) = (\mathbf{p}_{i^n}^0, \mathbf{R}_{i^n}^0)(\mathbf{x}_{i^n}^n), \quad n\in \mathcal{N}_S^D,\\
\label{eq:compat_GEB_transmi}
&\mathbf{p}_i^0(\mathbf{x}_i^n)  = \mathbf{p}_{i^n}^0(\mathbf{x}_{i^n}^n), \quad (\mathbf{R}_i^0 R_i^\intercal)(\mathbf{x}_i^n)  = (\mathbf{R}_{i^n}^0 R_{i^n}^\intercal)(\mathbf{x}_{i^n}^n), \quad i\in \mathcal{I}^n, \, n \in \mathcal{N}_M,
\end{align}
\end{subequations}
\end{linenomath}
and
\begingroup % keep the change local
\setlength\arraycolsep{3pt}
\renewcommand*{\arraystretch}{0.9}
\begin{linenomath}
\begin{subequations}\label{eq:compat_GEB_01}
\begin{align}
\label{eq:compat_GEB_01_1}
&\mathbf{p}_i^1(\mathbf{x}_i^n) = \mathbf{p}_{i^n}^1(\mathbf{x}_{i^n}^n), \quad w_i^0(\mathbf{x}_i^n) = w_{i^n}^0(\mathbf{x}_{i^n}^n), \quad i \in \mathcal{I}^n, \, n \in \mathcal{N}_M\\
\label{eq:compat_GEB_01_2}
&\sum_{i \in \mathcal{I}^n} \tau_i^n \left(\overline{R}_i \mathbf{C}_i^{-1} \begin{bmatrix}
(\mathbf{R}_i^0)^{\intercal} \frac{\mathrm{d}}{\mathrm{d}x} \mathbf{p}_i^0 - e_1\\
\mathrm{vec}\left( (\mathbf{R}_i^0)^{\intercal} \frac{\mathrm{d}}{\mathrm{d}x}\mathbf{R}_i^0 - R_i^{\intercal}\frac{\mathrm{d}}{\mathrm{d}x} R_i \right)
\end{bmatrix}\right)(\mathbf{x}_i^n) = q_n(0), \quad n \in \mathcal{N}_M,\\
\label{eq:compat_GEB_01_3}
&\tau_{i^n}^n \left(\overline{R}_{i^n} \mathbf{C}_{i^n}^{-1} \begin{bmatrix}
(\mathbf{R}_{i^n}^0)^{\intercal} \frac{\mathrm{d}}{\mathrm{d}x} \mathbf{p}_{i^n}^0 - e_1\\
\mathrm{vec}\left( (\mathbf{R}_{i^n}^0)^{\intercal} \frac{\mathrm{d}}{\mathrm{d}x}\mathbf{R}_{i^n}^0 - R_{i^n}^{\intercal}\frac{\mathrm{d}}{\mathrm{d}x} R_{i^n} \right)
\end{bmatrix} \right)(\mathbf{x}_{i^n}^n) = q_n(0), \quad n \in \mathcal{N}_S^N,\\
\label{eq:compat_GEB_01_4}
&\mathbf{p}_{i^n}^1(\mathbf{x}_{i^n}^n) = \frac{\mathrm{d}}{\mathrm{d}t}f_n^\mathbf{p}(0), \quad w_{i^n}^0(\mathbf{x}_{i^n}^n) = \frac{\mathrm{d}}{\mathrm{d}t}f_n^\mathbf{R}(0), \quad n \in \mathcal{N}_S^D.
\end{align}
\end{subequations}
\end{linenomath}
\endgroup

\begin{theorem} \label{thm:solGEB}
Consider a general network, and assume that:
\begin{enumerate}[label=(\roman*)]
\item \label{thm:solGEB_c1} the beam parameters $(\mathbf{M}_i, \mathbf{C}_i, R_i)$ and initial data $(\mathbf{p}_i^0, \mathbf{R}_i^0, \mathbf{p}_i^1, w_i^0)$ have the regularity \eqref{eq:reg_beampara} and \eqref{eq:reg_Idata_GEB}, and $y_i^0$ is the associated function defined by \eqref{eq:rel_inidata},

\item \label{thm:solGEB_c2} the Neumann data $f_n = f_n(t, \mathbf{R}_{i^n})$ are of the form \eqref{eq:def_fn}, for given functions $q_n$ of regularity \eqref{eq:reg_Ndata_IGEB},

\item \label{thm:solGEB_c3} the Dirichlet data $(f_n^\mathbf{p}, f_n^\mathbf{R})$  are of regularity \eqref{eq:reg_Ndata_GEB_D}, and $q_n$ are the associated functions defined by \eqref{eq:def_qn_D},

\item \label{thm:solGEB_c4} the compatibility conditions \eqref{eq:compat_GEB_-1_GEBtransmi} hold.
\end{enumerate}
Then, if there exists a unique solution $(y_i)_{i \in \mathcal{I}} \in \prod_{i=1}^N C^1([0, \ell_i]\times [0, T]; \mathbb{R}^{12})$ to \eqref{eq:syst_physical} with initial and nodal data $y_i^0$ and $q_n$ (for some $T>0$), there exists a unique solution
$(\mathbf{p}_i, \mathbf{R}_i)_{i \in \mathcal{I}} \in  \prod_{i=1}^N C^2([0, \ell_i]\times [0, T]; \mathbb{R}^3 \times \mathrm{SO}(3))$
to \eqref{eq:GEB_netw} with initial data $(\mathbf{p}_i^0, \mathbf{R}_i^0, \mathbf{p}_i^1, w_i^0)$ and nodal data $f_n$, $(f_n^\mathbf{p}, f_n^\mathbf{R})$, and $(y_i)_{i\in \mathcal{I}} = \mathcal{T}((\mathbf{p}_i, \mathbf{R}_i)_{i\in\mathcal{I}})$.
\end{theorem}

\begin{remark}
We have the following restriction on the form of the Neumann data $f_n$: it must be possible to express it as a function $q_n = q_n(t)$ in the body-attached basis (see Subsection \ref{subsec:GEBmodels}).
\end{remark}

The proof of Theorem \ref{thm:solGEB}, given in Section \ref{sec:invert_transfo}, consists in using the last six equations of \eqref{eq:IGEB_gov} as compatibility conditions to prove that the transformation $\mathcal{T}$, defined in \eqref{eq:transfo}, is bijective on some spaces (see Lemma \ref{lem:invert_transfo}); this relies on the use of quaternions \cite{chou1992} to parametrize the rotations matrices, and existence and uniqueness results for (seemingly overdetermined) first-order linear PDE systems. Once that this property of the transformation is established, one recovers notably the governing system \eqref{eq:GEB_gov} by using the first six equations of \eqref{eq:IGEB_gov}. The transmission conditions are recovered by first showing that the rigid joint assumption \eqref{eq:GEB_rigid_angles} is fulfilled and then deducing \eqref{eq:GEB_continuity_pi}-\eqref{eq:GEB_Kirchhoff} from \eqref{eq:IGEB_cont_velo}-\eqref{eq:IGEB_Kirchhoff}.

Corollary \ref{coro:wellposedGEB} below follows from Theorem \ref{th:existence} and Theorem \ref{thm:solGEB}.

\begin{corollary}
\label{coro:wellposedGEB}
Consider a general network and suppose that the conditions \ref{thm:solGEB_c1}-\ref{thm:solGEB_c2}-\ref{thm:solGEB_c3}-\ref{thm:solGEB_c4} of Theorem \ref{thm:solGEB} are fulfilled, suppose that the beam parameters $(\mathbf{M}_i, \mathbf{C}_i)$ satisfy Assumption \ref{as:mass_flex}, and that the compatibility conditions \eqref{eq:compat_GEB_01} hold.
Then, for any $T>0$, there exists $\varepsilon_0>0$ such that for all $\varepsilon \in (0, \varepsilon_0)$, and for some $\delta>0$, if moreover $\|y_i^0\|_{C_x^1}+ \|q_n\|_{C_t^1}\leq \delta$, then there exists a unique solution $(\mathbf{p}_i, \mathbf{R}_i)_{i \in \mathcal{I}} \in \prod_{i=1}^N C^2([0, \ell_i]\times[0, T]; \mathbb{R}^3 \times \mathrm{SO}(3))$ to \eqref{eq:GEB_netw} with initial data $(\mathbf{p}_i^0, \mathbf{R}_i^0, \mathbf{p}_i^1, w_i^0)$ and nodal data $f_n$, $(f_n^\mathbf{p}, f_n^\mathbf{R})$.
\end{corollary}

\begin{remark}
Under \eqref{eq:compat_GEB_-1_GEBtransmi}, the conditions \eqref{eq:compat_GEB_01} are just an equivalent way of imposing that $y_i^0$ fulfills the first-order compatibility conditions of \eqref{eq:syst_physical}, but expressed in terms of the data of the GEB model.
\end{remark}

Finally, from Theorems \ref{th:controllability} and \ref{thm:solGEB}, one obtains Corollary \ref{coro:controlGEB} below.

\begin{corollary} \label{coro:controlGEB}
Consider the A-shaped network defined by \eqref{eq:A_netw}, and assume that

\begin{enumerate}[label=(\roman*)]
\item the beam parameters $(\mathbf{M}_i, \mathbf{C}_i, R_i)$ and initial data $(\mathbf{p}_i^0, \mathbf{R}_i^0, \mathbf{p}_i^1, w_i^0)$ have the regularity \eqref{eq:reg_beampara} and \eqref{eq:reg_Idata_GEB}, the former satisfy Assumption \ref{as:mass_flex} and the latter fulfill \eqref{eq:compat_GEB_transmi}, and $y_i^0$ is the associated function defined by \eqref{eq:rel_inidata},

\item the Neumann data $f_n = f_n(t, \mathbf{R}_{i^n})$, for $n \in \{1, 2, 3\}$ are of the form \eqref{eq:def_fn}, for given functions $q_n$ of regularity \eqref{eq:reg_Ndata_IGEB},

\item the compatibility conditions \eqref{eq:compat_GEB_01_1}-\eqref{eq:compat_GEB_01_2} for all $n \in \{1, 2, 3\}$ hold.
\end{enumerate}
Let $\overline{T}>0$ be defined by \eqref{eq:minT}. Then, for any $T>T^*>\overline{T}$, there exists $\varepsilon_0>0$ such that for all $\varepsilon\in (0, \varepsilon_0)$, for some $\delta, \gamma>0$, and for any nodal profiles $\overline{y}_1, \overline{y}_2 \in C^1([T^*, T]; \mathbb{R}^{12})$ satisfying $\|\overline{y}_i\|_{C_t^1}\leq \gamma$ and the  transmission conditions \eqref{eq:IGEB_cont_velo}-\eqref{eq:IGEB_Kirchhoff} at the node $n=1$, if additionally $\|y_i^0\|_{C_x^1} + \|f_n\|_{C_t^1} \leq \delta$ ($i \in \mathcal{I}, \ n \in \{1, 2, 3\}$), then there exist controls $f_4, f_5 \in C^1([0, T]; \mathbb{R}^6)$ with $\|f_n\|_{C_t^1}\leq \varepsilon$ such that System \eqref{eq:GEB_netw} with initial data $(\mathbf{p}_i^0, \mathbf{R}_i^0, \mathbf{p}_i^1, w_i^0)$ and boundary data $(f_n)_{n \in \{1, 2, 3\}}$, admits a unique solution $(\mathbf{p}_i, \mathbf{R}_i)_{i \in \mathcal{I}} \in \prod_{i=1}^N C^2([0, \ell_i]\times[0, T]; \mathbb{R}^3 \times \mathrm{SO}(3))$, and $(y_i)_{i\in\mathcal{I}} := \mathcal{T}((\mathbf{p}_i, \mathbf{R}_i)_{i\in\mathcal{I}})$ fulfills $\|y_i\|_{C_{x,t}^1}\leq \varepsilon$ and the nodal profiles \eqref{eq:aim}.
\end{corollary}

\begin{remark}
In Corollary \ref{coro:controlGEB},
\begin{enumerate}
\item the profiles given at the node $n=1$ affect the intrinsic variables $\mathcal{T}_i(\mathbf{p}_i, \mathbf{R}_i)$, for $i \in \{1, 2\}$, and not directly the displacements and rotations $(\mathbf{p}_i, \mathbf{R}_i)$;

\item for $i \in \{4,5\}$ the control $f_i$ is given by \eqref{eq:def_fn} where $q_i$ is the control provided by Theorem \ref{th:controllability} for System \eqref{eq:syst_physical}. The smallness of the $C^1$ norm of $f_i$ comes from a combination of the fact that $q_i$ and $y_i$ (and thus, as can be seen in \eqref{eq:form_yi}, also the angular velocity $W_i$) have small $C^1$ norms, and that the expression of $f_i$ and $\frac{\mathrm{d}}{\mathrm{d}t}f_i$ involves only the functions $q_i, W_i$ and the unitary matrices $\mathbf{R}_i, R_i$. Indeed, $f_i = \mathrm{diag}\big(\mathbf{R}_i(\ell_i, \cdot), \mathbf{R}_i(\ell_i, \cdot)\big)  q_i$ and one may compute that
\begin{linenomath}
\begin{align*}
\frac{\mathrm{d}}{\mathrm{d}t} f_i = \mathrm{diag}\big(\mathbf{R}_i(\ell_i, \cdot), \mathbf{R}_i(\ell_i, \cdot)\big) \frac{\mathrm{d}}{\mathrm{d}t} q_i +  \mathrm{diag}\big((\mathbf{R}_i \widehat{W}_i)(\ell_i, \cdot), (\mathbf{R}_i \widehat{W}_i)(\ell_i, \cdot)\big)  q_i.
\end{align*}
\end{linenomath}
\end{enumerate}
\end{remark}

\section{Existence and uniqueness for the IGEB network}
\label{sec:exist}

We now turn to the proof of Theorem \ref{th:existence}.

\subsection{Hyperbolicity of the system}
\label{subsec:hyperbolic}

Let $T >0$, $i\in \mathcal{I}$ and $x \in [0, \ell_i]$. One may quickly verify that the matrix $A_i(x)$, defined in \eqref{eq:def_Ai}, has only real eigenvalues: six positive ones which are the square roots of the eigenvalues of $\Theta_i(x)$ (defined in Assumption \ref{as:mass_flex}), and six negative ones which are equal to the former but with a minus sign.
Furthermore, some computations yield the following lemma whose proof is given in \cite[Section 4]{R2020}.

\begin{lemma}
Suppose that Assumption \ref{as:mass_flex} is fulfilled and, for any $i \in \mathcal{I}$, let $U_i$, $D_i \in C^2([0, \ell_i]; \mathbb{R}^{6\times 6})$ be the functions introduced in Assumption \ref{as:mass_flex}.
Then, $A_i \in C^2([0, \ell_i]; \mathbb{R}^{12\times 12})$ may be diagonalized as follows. One has $A_i = L_i^{-1} \mathbf{D}_i L_i$ in $[0, \ell_i]$, where $\mathbf{D}_{i}$, $L_i \in C^2( [0, \ell_i]; \mathbb{R}^{12 \times 12})$ are defined by
\begin{linenomath}
\begin{equation} \label{eq:def_bfDi_Li}
\mathbf{D}_i = \mathrm{diag}(-D_{i}, D_{i}), \qquad L_i = \begin{bmatrix}
U_i \mathbf{C}_i^{-\sfrac{1}{2}} & D_{i}U_i \mathbf{C}_i^{\sfrac{1}{2}} \\
U_i \mathbf{C}_i^{-\sfrac{1}{2}} & - D_{i}U_i \mathbf{C}_i^{\sfrac{1}{2}}
\end{bmatrix},
\end{equation}
\end{linenomath}
and the inverse $L_i^{-1} \in C^2( [0, \ell_i]; \mathbb{R}^{12 \times 12})$ is given by
\begin{linenomath}
\begin{align} \label{eq:inverseLi}
L_i^{-1} = \frac{1}{2} \begin{bmatrix}
\mathbf{C}_i^{\sfrac{1}{2}} U_i^\intercal & \mathbf{C}_i^{\sfrac{1}{2}} U_i^\intercal \\
\mathbf{C}_i^{-\sfrac{1}{2}} U_i^\intercal D_i^{-1} & - \mathbf{C}_i^{-\sfrac{1}{2}} U_i^\intercal D_i^{-1}
\end{bmatrix}.
\end{align}
\end{linenomath}
\end{lemma}

\subsection{Change of variable to Riemann invariants} 
\label{subsec:change_var}

Now, we can write \eqref{eq:syst_physical} in diagonal form by applying the change of variable 
\begin{linenomath}
\begin{align} \label{eq:change_var_Li}
r_i(x,t) = L_i(x) y_i(x,t), \qquad \text{for all }x \in [0, \ell_i], \ t \in [0, T], \ i \in \mathcal{I}.
\end{align}
\end{linenomath}
The first (resp. last) six components of $r_i$ correspond to the negative (resp. positive) eigenvalues of $A_i$, thus, for all $i \in \mathcal{I}$, we denote
\begin{linenomath}
\begin{align*}
r_i = \begin{bmatrix}
r_i^-\\
r_i^+
\end{bmatrix}, \qquad r_i^-,\, r_i^+ \colon [0, \ell_i]\times [0, T] \rightarrow \mathbb{R}^6.
\end{align*}
\end{linenomath}
In addition, in order to write the transmission conditions concisely, we introduce the invertible matrix $\gamma_i^n$ and positive definite symmetric matrix $\sigma_i^n$
\begin{linenomath}
\begin{align*}
\gamma_i^n &= (\overline{R}_i  \mathbf{C}_i^{\sfrac{1}{2}} U_i^\intercal)(\mathbf{x}_i^n), \qquad
\sigma_i^n = (\overline{R}_i  \mathbf{C}_i^{-\sfrac{1}{2}} U_i^\intercal D_i^{-1} U_i \mathbf{C}_i^{-\sfrac{1}{2}} \overline{R}_i^\intercal)(\mathbf{x}_i^n)
\end{align*}
\end{linenomath}
for all $n \in \mathcal{N}$ and $i \in \mathcal{I}^n$.
Notice that $\sigma_i^n \gamma_i^n = \overline{R}_i(\mathbf{x}_i^n) \mathbf{C}_i^{-\sfrac{1}{2}} U_i^\intercal D_i^{-1}$.

\medskip

\noindent Then, taking \eqref{eq:def_bfDi_Li}-\eqref{eq:inverseLi} into account, the system obtained by applying the change of variable \eqref{eq:change_var_Li} to System \eqref{eq:syst_physical} reads
\begin{linenomath}
\begin{subnumcases}{\label{eq:syst_diagonal}}
\label{eq:r_IGEB_gov}
\partial_t r_i + \mathbf{D}_i \partial_x r_i + B_i r_i = g_i(\cdot, r_i), &\hspace{-0.45cm}$\text{in } (0, \ell_i)\times(0, T), \, i \in \mathcal{I}$\\
\label{eq:r_IGEB_cont_velo}
\gamma_i^n (r_i^- + r_i^+)(\mathbf{x}_i^n, t) &\nonumber \vspace{-0.2cm}\\
\qquad \quad = \gamma_{i^n}^n (r_{i^n}^- + r_{i^n}^+)(\mathbf{x}_{i^n}^n, t), &\hspace{-0.45cm}$t \in (0, T), \, i \in \mathcal{I}^n, \, n \in \mathcal{N}_M$\\
\label{eq:r_IGEB_Kirchhoff}
\sum_{i\in\mathcal{I}^n} \frac{\tau_i^n}{2} \sigma_i^n \gamma_i^n (r_i^- -  r_i^+)(\mathbf{x}_i^n, t) = q_n(t), &\hspace{-0.45cm}$t \in (0, T), \, n \in \mathcal{N}_M$\\
\label{eq:r_IGEB_condNSz}
(r_{i^n}^- - r_{i^n}^+)(\mathbf{x}_{i^n}^n, t) &\vspace{-0.2cm} \nonumber\\
\qquad \quad = 2 \tau_{i^n}^n (D_{i^n} U_{i^n}\mathbf{C}_{i^n}^{\sfrac{1}{2}})(\mathbf{x}_{i^n}^n) q_n(t), &\hspace{-0.45cm}$t \in (0, T), \, n \in \mathcal{N}_S^N$\\
\label{eq:r_IGEB_condNSv}
(r_{i^n}^- + r_{i^n}^+)(\mathbf{x}_{i^n}^n, t) &\vspace{-0.2cm} \nonumber\\
\qquad \quad = 2 (U_{i^n} \mathbf{C}_{i^n}^{- \sfrac{1}{2}})(\mathbf{x}_{i^n}^n) q_n(t), &\hspace{-0.45cm}$t \in (0, T), \, n \in \mathcal{N}_S^D$\\
\label{eq:r_IGEB_ini_cond}
r_i(x, 0) = r_i^0(x), &\hspace{-0.45cm}$x \in (0, \ell_i), \, i \in \mathcal{I}$.
\end{subnumcases}
\end{linenomath}
In the governing system \eqref{eq:r_IGEB_gov}, the coefficient $B_i \in C^1([0, \ell_i]; \mathbb{R}^{12\times 12})$ is defined by $B_i(x) = L_i(x) \overline{B}_i(x) L_i(x)^{-1} + L_i(x) A_i(x) \frac{\mathrm{d}}{\mathrm{d}x}L_i^{-1}(x)$, while the source is defined by $g_i(x,u) = L_i(x) \overline{g}_i(x,L_i(x)^{-1} u)$ for all $i \in \mathcal{I}$, $x \in [0, \ell_i]$ and $u \in \mathbb{R}^{12}$. The corresponding initial data in \eqref{eq:r_IGEB_ini_cond} for this system is $r_i^0 = L_i y_i^0$.

\subsection{Outgoing and incoming information}
\label{subsec:out_in_info}

\begin{figure}
  \begin{subfigure}{0.6\textwidth}
  \centering
    \includegraphics[scale=0.75]{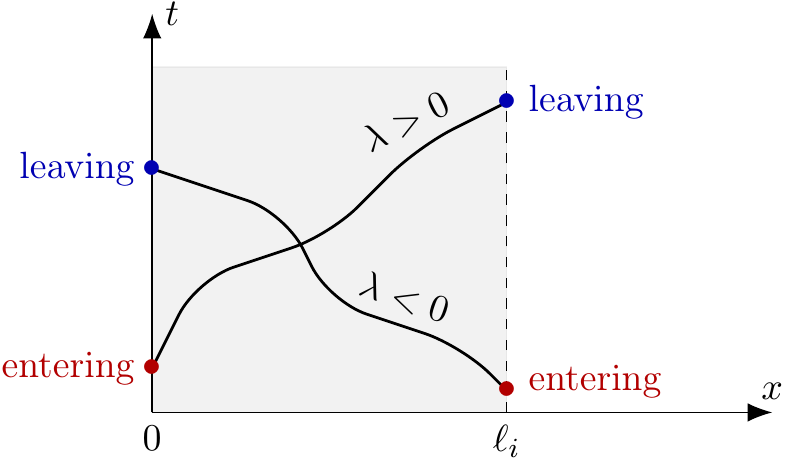}
\caption{Characteristic curves $(\mathbf{x}(t), t)$ with $\frac{\mathrm{d}\mathbf{x}}{\mathrm{d}t}(t) = \lambda(\mathbf{x}(t))$, where either $\lambda(s)>0$ or $\lambda(s)<0$ for all $s \in [0, \ell_i]$.}
\label{fig:charac}
  \end{subfigure}%
  \hspace*{\fill}   % maximize separation between the subfigures
  \begin{subfigure}{0.4\textwidth}
    \centering
\includegraphics[scale=0.7]{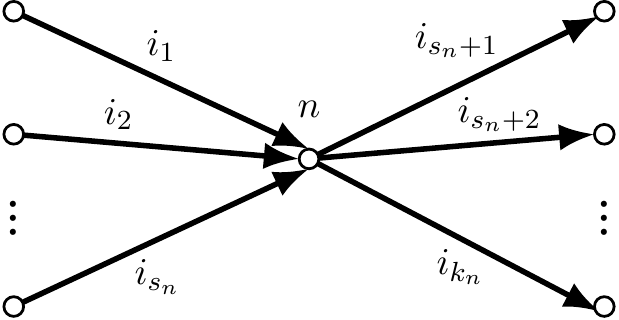}
    \caption{Form of $\mathcal{I}_n$ at a node $n$.}
    \label{fig:NM_notation}
  \end{subfigure}

\caption{Outgoing and incoming information.}
\label{fig:out_in_info}
\end{figure}

For any $n \in \mathcal{N}$, let us denote by $s_n \in \{0, \ldots, k_n\}$  (resp. by $k_n-s_n$) the number of beams ending (resp. starting) at the node $n$; see Fig. \ref{fig:NM_notation}. More precisely, we suppose that 
\begin{linenomath}
\begin{align*}
\mathcal{I}^n = \{i_1, \ldots, i_{k_n}\} \quad \text{with} \quad i_1 < i_2 < \ldots < i_{s_n} \quad \text{and} \quad i_{s_n+1} < i_{s_n+2} < \ldots < i_{k_n},
\end{align*}
\end{linenomath}
and that $\tau_{i_\alpha}^n = -1$ for all $\alpha \in \{1, \ldots, s_n\}$, while $\tau_{i_\alpha}^n = +1$ for all $\alpha \in \{s_n+1, \ldots, k_n\}$. This is not to be confused with the notation $i^n$ introduced in \eqref{eq:def_in}.

\medskip

\noindent For any node $n$ and any incident edge $i \in \mathcal{I}^n$, we call \emph{outgoing} (resp. \emph{incoming}) \emph{information}, the components of $r_i$ which correspond to characteristics entering (resp. leaving) the domain $[0, \ell_i]\times[0, +\infty)$ at this node (see Fig. \eqref{fig:charac}).

Namely, here, the outgoing (resp. incoming) information at the node $n$ is $r_{i_\alpha}^-(\ell_{i_\alpha}, t)$ (resp. $r_{i_\alpha}^+(\ell_{i_\alpha}, t)$) for all $\alpha \in \{1, \ldots, s_n\}$, and $r_{i_k}^+(0, t)$ (resp. $r_{i_k}^-(0, t)$) for all $k \in \{s_n+1, \ldots, k_n\}$.
We then define the functions $r_n^\mathrm{out}, r_n^\mathrm{in} \colon [0, T] \rightarrow \mathbb{R}^{6k_n}$ by
\begingroup % keep the change local
\setlength\arraycolsep{3pt}
\renewcommand*{\arraystretch}{0.9}
\begin{linenomath}
\begin{align*}
r_n^\mathrm{out}(t) = \begin{bmatrix}
r_{i_1}^-(\ell_{i_1}, t)\\
\vdots\\
r_{i_{s_n}}^-(\ell_{i_{s_n}}, t)\\
r_{i_{s_n+1}}^+(0, t) \\
\vdots\\
r_{i_{k_n}}^+(0, t)
\end{bmatrix}, \qquad r_n^\mathrm{in}(t) = \begin{bmatrix}
r_{i_1}^+(\ell_{i_1}, t)\\
\vdots\\
r_{i_{s_n}}^+(\ell_{i_{s_n}}, t)\\
r_{i_{s_n+1}}^-(0, t) \\
\vdots\\
r_{i_{k_n}}^-(0, t)
\end{bmatrix}.
\end{align*}
\end{linenomath}
\endgroup
We also denote $r_n^\mathrm{out} = ((r_{n,1}^\mathrm{out})^\intercal, \ldots, (r_{n,k_n}^\mathrm{out})^\intercal)$, where $r_{n,\alpha}^\mathrm{out}(t) \in \mathbb{R}^6$ for all $\alpha \in \{1, \ldots, k_n\}$; a similar notation is used for $r_n^\mathrm{in}$.

\medskip

\noindent Taking into account this notation, and the sign of $\tau_i^n$, we observe that the Kirchhoff condition \eqref{eq:r_IGEB_Kirchhoff} is equivalent to
\begin{linenomath}
\begin{align*}
-\sum_{\alpha =1}^{s_n} \sigma_{i_\alpha}^n \gamma_{i_\alpha}^n (r_{i_\alpha}^- -  r_{i_\alpha}^+)(0, t) + \sum_{k=s_n+1}^{k_n} \sigma_{i_k}^n \gamma_{i_k}^n (r_{i_k}^- -  r_{i_k}^+)(\ell_{i_k}, t) = 2q_n(t),
\end{align*}
\end{linenomath}
which can also be written in the form
\begin{linenomath} 
\begin{align*}
\sum_{\alpha=1}^{k_n} \sigma_{i_\alpha}^n \gamma_{i_\alpha}^n r_{n,\alpha}^\mathrm{out}(t) = \sum_{\alpha=1}^{k_n} \sigma_{i_\alpha}^n \gamma_{i_\alpha}^n r_{n,\alpha}^\mathrm{in}(t) + 2q_n(t).
\end{align*}
\end{linenomath}
The continuity condition \eqref{eq:r_IGEB_cont_velo} is equivalent to
\begin{linenomath}
\begin{align*}
\gamma_{i_1}^n (r_{i_1}^- + r_{i_1}^+)(\mathbf{x}_{i_1}^n, t) = \gamma_{i_\alpha}^n (r_{i_\alpha}^- + r_{i_\alpha}^+)(\mathbf{x}_{i_\alpha}^n, t) \quad \text{for all }\alpha \in \{2, \ldots, k_n\}
\end{align*}
\end{linenomath}
%%%% additional detail
%%%\begin{align*}
%%%&\gamma_{i_1}^n (r_{i_1}^- + r_{i_1}^+)(\ell_{i_1}, t) =
%%%\begin{cases}
%%%\gamma_{i_j}^n (r_{i_j}^- + r_{i_j}^+)(\ell_{i_j}, t) & \text{for all }j \in \{2, \ldots, s_n\}\\
%%%\gamma_{i_j}^n (r_{i_j}^- + r_{i_j}^+)(0, t) & \text{for all }j\in \{s_n+1, \ldots, k_n\}
%%%\end{cases}
%%%&& \text{if } s_n\geq 1\\
%%%&\gamma_{i_1}^n (r_{i_1}^- + r_{i_1}^+)(0, t) =  \gamma_{i_j}^n (r_{i_j}^- + r_{i_j}^+)(0, t),&& \text{if } s_n=0.
%%%\end{align*}
which can be seen to also write as 
\begin{linenomath}
\begin{align*}
\gamma_{i_1}^n r_{n,1}^\mathrm{out}(t) - \gamma_{i_\alpha}^n r_{n,\alpha}^\mathrm{out}(t) = - \gamma_{i_1}^n r_{n,1}^\mathrm{in}(t) + \gamma_{i_\alpha}^n r_{n,\alpha}^\mathrm{in}(t) \quad \text{for all }\alpha \in \{2, \ldots, k_n\}.
\end{align*}
\end{linenomath}
Hence, at any multiple node $n$, the transmission conditions \eqref{eq:r_IGEB_Kirchhoff}-\eqref{eq:r_IGEB_cont_velo} are equivalent to the following system:
\begin{linenomath}
\begin{align*}
\mathbf{A}_n \mathbf{G}_n r_n^\mathrm{out}(t) = \mathbf{B}_n \mathbf{G}_n r_n^\mathrm{in}(t) + \begingroup % keep the change local
\setlength\arraycolsep{3pt}
\renewcommand*{\arraystretch}{0.8}
\begin{bmatrix}
2q_n(t) \\ \mathbf{0}_{6k_n-6, 1}
\end{bmatrix},
\endgroup
\end{align*}
\end{linenomath}
where $\mathbf{A}_n, \mathbf{B}_n, \mathbf{G}_n \in \mathbb{R}^{6k_n \times 6k_n}$ are defined by
\begin{linenomath}
\begin{align*}
\mathbf{A}_n = \begin{bmatrix}
\mathbf{a}_n & \mathbf{b}_n\\
\mathbf{c}_n & \mathbf{I}_{6(k_n-1)}
\end{bmatrix}, \quad
\mathbf{B}_n = \begin{bmatrix}
\mathbf{a}_n & \mathbf{b}_n\\
-\mathbf{c}_n & -\mathbf{I}_{6(k_n-1)}
\end{bmatrix}, \quad \mathbf{G}_n = \mathrm{diag}(\gamma_{i_1}^n, \ldots, \gamma_{i_{k_n}}^n),
\end{align*}
\end{linenomath}
the sub-matrices $\mathbf{a}_n \in \mathbb{R}^{6\times 6}$, $\mathbf{b}_n \in \mathbb{R}^{6\times 6(k_n-1)}$ and $\mathbf{c}_n \in \mathbb{R}^{6(k_n-1) \times 6}$ being defined by $\mathbf{a}_n = \sigma_{i_1}^n$, $\mathbf{b}_n = \big[ \sigma^n_{i_2} \ \sigma^n_{i_3} \ \ldots \   \sigma^n_{i_{k_n}}\big]$ and $\mathbf{c}_n = - \big[\mathbf{I}_6 \ \mathbf{I}_6 \ \ldots \ \mathbf{I}_6 \big]^\intercal$.

The matrix $\mathbf{G}_n$ is clearly invertible and one can check that $\mathbf{A}_n$ is also invertible (using the same reasoning as \cite[Lemma 4.4]{R2020}).
%%% Here is more detail:
%%%{\color{blue!50!black} Then, 
%%%\begin{align*}
%%%r_\mathrm{out}^n(t) = \mathbf{G}_n^{-1} \mathbf{J}_n \mathbf{G}_n r_\mathrm{in}^n(t) + \mathbf{G}_n^{-1} \mathbf{A}_n^{-1}\begin{bmatrix}
%%%q_n(t) \\ \mathbf{0}_{6(k_n-1), 1}
%%%\end{bmatrix},
%%%\end{align*}
%%%where $\mathbf{J}_n \in \mathbb{R}^{6k_n \times 6k_n}$ is defined by $\mathbf{J}_n = \mathbf{A}_n^{-1} \mathbf{B}_n$, which also takes the form
%%%\begin{align*}
%%%\mathbf{J}_n = 2 
%%%\left[ \begin{smallmatrix}
%%%\sigma^n & & \mathbf{0}_6\\
%%% & \ddots &  \\
%%%\mathbf{0}_6 &  & \sigma^n
%%%\end{smallmatrix} \right]
%%%\left[ \begin{smallmatrix}
%%%\mathbf{I}_6 & \ldots & \mathbf{I}_6\\
%%%\vdots & \ddots & \vdots \\
%%%\mathbf{I}_6 & \ldots & \mathbf{I}_6
%%%\end{smallmatrix} \right]
%%%\mathrm{diag}(\sigma_{i_1}^n, \ldots, \sigma_{i_{k_n}}^n) - \mathbf{I}_{6k_n}
%%%\end{align*}
%%%for $\sigma^n := \sum_{j=1}^{k_n} \sigma_{i_j}^n$. MAYBE REMOVE UNNECESSARY DETAILS HERE}
For any $n \in \mathcal{N}$, let us define $\mathcal{B}_n  \in \mathbb{R}^{6k_n \times 6 k_n}$ by
\begin{linenomath}
\begin{align*}
\mathcal{B}_n = \left\{\begin{aligned}
&\mathbf{G}_n^{-1}  \mathbf{A}_n^{-1} \mathbf{B}_n \mathbf{G}_n && n \in \mathcal{N}_M\\
&\mathbf{I}_6 && n \in \mathcal{N}_S^N\\
&-\mathbf{I}_6 && n \in \mathcal{N}_S^D,
\end{aligned}\right.
\end{align*}
\end{linenomath}
as well as $\mathcal{Q}_n \in \mathbb{R}^{6k_n \times 6 k_n}$ and $\mathbf{q}_n \in C^1([0, T]; \mathbb{R}^{6k_n})$ by
\begin{linenomath}
\begin{align*}
\mathcal{Q}_n = \left\{\begin{aligned}
&2 \mathbf{G}_n^{-1}  \mathbf{A}_n^{-1} && n \in \mathcal{N}_M\\
&2 \tau_{i^n}^n (D_{i^n} U_{i^n}\mathbf{C}_{i^n}^{\sfrac{1}{2}})(\mathbf{x}_{i^n}^n) && n \in \mathcal{N}_S^N\\
&2 (U_{i^n} \mathbf{C}_{i^n}^{- \sfrac{1}{2}})(\mathbf{x}_{i^n}^n)  && n \in \mathcal{N}_S^D,
\end{aligned}\right. \quad
\mathbf{q}_n(t) = \begin{cases}
\begingroup % keep the change local
\setlength\arraycolsep{3pt}
\renewcommand*{\arraystretch}{0.9}
\begin{bmatrix}
q_n(t) \\ \mathbf{0}_{6k_n-6, 1}
\end{bmatrix}\endgroup & n \in \mathcal{N}_M\\
q_n(t) & n \in \mathcal{N}_S.
\end{cases}
\end{align*}
\end{linenomath}
Then, System \eqref{eq:syst_diagonal} also reads
\begin{linenomath}
\begin{align*}
\begin{dcases}
\partial_t r_i + \mathbf{D}_i(x) \partial_x r_i + B_i(x) r_i = g_i(x, r_i) &\text{in } (0, \ell_i)\times(0, T), \, i \in \mathcal{I}\\
r^\mathrm{out}_n(t) = \mathcal{B}_n r^\mathrm{in}_n(t) + \mathcal{Q}_n \mathbf{q}_n(t) & t \in (0, T), \, n \in \mathcal{N}\\
r_i(x, 0) = r_i^0(x) & x \in (0, \ell_i), \, i \in \mathcal{I}.
\end{dcases}
\end{align*}
\end{linenomath}

\subsection{Proof of Theorem \ref{th:existence}}
\label{subsec:proof_exist}

Relying upon Subsections \ref{subsec:hyperbolic}, \ref{subsec:change_var} and \ref{subsec:out_in_info}, and \cite{li2010controllability, wang2006exact}, we now prove Theorem \ref{th:existence}.

\begin{proof}[Proof of Theorem \ref{th:existence}]

The local and semi-global existence and uniqueness of $C_{x,t}^1$ solutions to general one-dimensional quasilinear hyperbolic systems have been addressed in \cite[Lem. 2.3, Th. 2.1]{wang2006exact}, which is an extension of \cite[Lem. 2.3, Th. 2.5]{li2010controllability} to nonautonomous systems.

Such results may be applied to the network system \eqref{eq:syst_physical}, since it can be written as a single larger hyperbolic system. One needs only to apply the change of variable $\widetilde{r}_i(\xi, t) = r_i(\ell_i \ell^{-1} \xi, t)$ for all $i \in \mathcal{I}$, $\xi \in [0, \ell]$ and $t \in [0, T]$ for some $\ell>0$, in order to make the spatial domain identical for all beams, and consider the larger $\mathbb{R}^{12N}$-valued unknown $\widetilde{r} = (\widetilde{r}_1^{\,\intercal}, \ldots, \widetilde{r}_N^{\, \intercal})^\intercal$. Then, $\widetilde{r}$ is governed by
\begin{linenomath}
\begin{align} \label{eq:single_hyperb_syst}
\begin{cases}
\partial_t \widetilde{r} + \widetilde{\mathbf{D}}(\xi) \partial_\xi \widetilde{r} + \widetilde{B}(\xi) \widetilde{r} = \widetilde{g}(\widetilde{r}) & \text{in }(0, \ell)\times(0, T)\\
\widetilde{r}^\mathrm{\, out}(t) = \widetilde{\mathcal{B}} \, \widetilde{r}^\mathrm{\, in}(t) + \widetilde{Q}\widetilde{\mathbf{q}}(t) & t \in (0, T)\\
\widetilde{r}(\xi, 0) = \widetilde{r}^0(\xi) & \xi \in (0, \ell),
\end{cases}
\end{align}
\end{linenomath}
where $\widetilde{\mathbf{D}}, \widetilde{B}, \widetilde{\mathcal{B}}, \widetilde{\mathcal{Q}}, \widetilde{\mathbf{q}}, \widetilde{r}^\mathrm{out}, \widetilde{r}^\mathrm{in}, \widetilde{r}^0$ and $\widetilde{g}$ are defined by
\begin{linenomath}
\begin{align*}
&\widetilde{\mathbf{D}}(\cdot) = \ell \mathrm{diag}\left(\ell_1^{-1}\mathbf{D}_1(\ell_1 \ell^{-1} \cdot), \ldots, \ell_N^{-1} \mathbf{D}_N(\ell_N \ell^{-1} \cdot) \right),\\
&\widetilde{B}(\cdot) = \mathrm{diag}\left(B_1(\ell_1 \ell^{-1} \cdot), \ldots, B_N(\ell_N \ell^{-1} \cdot)\right),\\
&\widetilde{\mathcal{B}} = \mathrm{diag}\left(\mathcal{B}_1, \ldots, \mathcal{B}_{\#\mathcal{N}}\right), \quad \widetilde{\mathcal{Q}} = \mathrm{diag}\left(\mathcal{Q}_1, \ldots, \mathcal{Q}_{\#\mathcal{N}}\right), \quad \widetilde{\mathbf{q}} = (\mathbf{q}_1^\intercal, \ldots, \mathbf{q}_{\# \mathcal{N}}^\intercal)^\intercal,\\
&\widetilde{r}^\mathrm{out} = \left((r_1^\mathrm{out})^\intercal, \ldots, (r_{\#\mathcal{N}}^\mathrm{out})^\intercal\right)^\intercal, \quad \widetilde{r}^\mathrm{in} = \left((r_1^\mathrm{in})^\intercal, \ldots, (r_{\#\mathcal{N}}^\mathrm{in})^\intercal\right)^\intercal,\\
&\widetilde{r}^0(\cdot) = \left(r^0(\ell_1 \ell^{-1} \cdot)^\intercal, \ldots, r^0(\ell_N, \ell^{-1} \cdot)^\intercal \right)^\intercal\\
&\widetilde{g}(\cdot, \mathbf{u}) = \left(\widetilde{g}_1(\ell_1 \ell^{-1} \cdot, \mathbf{u}_1)^\intercal, \ldots, \widetilde{g}_N(\ell_N \ell^{-1} \cdot, \mathbf{u}_N)^\intercal \right)^\intercal,
\end{align*}
\end{linenomath}
where we denoted $\mathbf{u} = (\mathbf{u}_1^\intercal, \ldots, \mathbf{u}_N^\intercal)^\intercal$ with $\mathbf{u}_i \in \mathbb{R}^{12}$ for all $i \in \mathcal{I}$.

Due to Subsection \ref{subsec:out_in_info}, the boundary conditions of \eqref{eq:single_hyperb_syst} are directly written in such a way that the outgoing information for System \eqref{eq:single_hyperb_syst} is a function of the incoming information, a sufficient criteria in \cite{li2010controllability, wang2006exact} to deduce well-posedness of the system.
\end{proof}

\section{Controllability of nodal profiles for the IGEB network}
\label{sec:controllability}

We now consider the A-shaped network defined by \eqref{eq:A_netw} and our aim is to prove Theorem \ref{th:controllability}. As pointed out in Section \ref{sec:model_results}, we will solve several forward and sidewise problems for \eqref{eq:IGEB_gov} (see Steps 1.3, 1.4, 1.5). The existence and uniqueness of semi-global in time solutions to these problems is provided by \cite{li2010controllability, wang2006exact}, as in Section \ref{sec:exist} for the overall network.

\begin{proof}[Proof of Theorem \ref{th:controllability}]
The proof is divided in three steps.
We start by constructing a solution satisfying all transmission conditions and the nodal profiles.
The choice of $\overline{T}$ (see \eqref{eq:minT}), and thus $T^*$, is explained in Step 2.

\medskip

\noindent \textit{Step 1.1 (see Fig. \ref{fig:A} top-left).}
Consider the forward problem for the entire network until time $\overline{T}$, where at the simple nodes $n \in \{4,5\}$, the controls $q_4, q_5$ are replaced by any functions $\overline{q}_4, \overline{q}_5 \in C^1([0; \overline{T}]; \mathbb{R}^6)$ satisfying the first-order compatibility conditions of \eqref{eq:syst_physical}. By Theorem \ref{th:existence}, for any $\gamma>0$ small enough, there exists $\delta>0$ such that \eqref{eq:syst_physical} admits a unique solution $(y_i^f)_{i\in \mathcal{I}} \in \prod_{i=1}^N C^1([0, \ell_i]\times[0, \overline{T}]; \mathbb{R}^{12})$ with $\|y_i^f\|_{C_{x,t}^1}\leq \gamma$, provided that $\|y_i^0\|_{C_x^1}+ \|q_n\|_{C_t^1}+\|\overline{q}_k\|_{C_t^1} \leq \delta$ for all $i \in \mathcal{I}, n \in \{1, 2, 3\}$ and $k \in \{4, 5\}$.

Similarly to the state $y_i$ (see \eqref{eq:form_yi}), we denote $y_i^f = ((v_i^f)^\intercal, (z_i^f)^\intercal)^\intercal$, and later on, we will also use such a notation for $\overline{y}_i$, $\overline{\overline{y}}_i$, $\widetilde{y}_i$ and $\mathbf{y}_i$.

\medskip

\noindent \textit{Step 1.2.}
At the node $n=1$, to obtain ``data'' $\overline{\overline{y}}_1, \overline{\overline{y}}_2 \in C^1([0, T])$ for the entire time interval with small $C^1$ norm and fulfilling the transmission conditions at this node, we connect $y_i^f$ (from Step 1.1), which is defined on $[0, \overline{T}]$, to the nodal profiles $\overline{y}_i$ defined on $[T^*, T]$ (see \eqref{eq:aim}).

We first find functions $\overline{\overline{v}}_1, \overline{\overline{z}}_1 \in C^1([0, T]; \mathbb{R}^6)$ \textcolor{black}{with $\|\overline{\overline{v}}_1\|_{C_t^1} + \|\overline{\overline{z}}_1\|_{C_t^1} \leq \gamma$} and such that
\begin{linenomath}
\begin{align} \label{eq:barbar_v1_z1}
\overline{\overline{v}}_1(t) = \left\{ \begin{aligned}
&v_1^f(0, t) && t \in [0, \overline{T}]\\
&\overline{v}_1(t) && t \in [T^*, T]
\end{aligned}\right., \quad \overline{\overline{z}}_1(t) = \left\{\begin{aligned}
&z_1^f(0, t) && t \in [0, \overline{T}]\\
&\overline{z}_1(t) && t \in [T^*, T]
\end{aligned}\right.,
\end{align}
\end{linenomath}
completing the gap between via, for example, cubic Hermite splines fulfilling the values and first derivatives prescribed by \eqref{eq:barbar_v1_z1} at $t = \overline{T}$ and $t = T^*$. The $C^1$ norm of such functions is bounded by that of $v_i^f, \overline{v}_i$ and $z_i^f,\overline{z}_i$, respectively.

Then, we define $\overline{\overline{v}}_2, \overline{\overline{z}}_2 \in C^1([0, T]; \mathbb{R}^6)$ by $\overline{\overline{v}}_2(t) = (\overline{R}_2^\intercal \overline{R}_1)(0) \overline{\overline{v}}_1(t)$ and $\overline{\overline{z}}_2(t) = - (\overline{R}_2^\intercal \overline{R}_1)(0) \overline{\overline{z}}_1(t)$, so that both the continuity and Kirchhoff conditions \eqref{eq:IGEB_cont_velo}-\eqref{eq:IGEB_Kirchhoff} are fulfilled.
Since $\overline{R}_i$ ($i\in\mathcal{I}$) is unitary and independent of time, one has $|\overline{\overline{v}}_1| = |\overline{\overline{v}}_2|$ and $|\overline{\overline{z}}_1| = |\overline{\overline{z}}_2|$, as well as $|\frac{\mathrm{d}}{\mathrm{d}t}\overline{\overline{v}}_1| = |\frac{\mathrm{d}}{\mathrm{d}t}\overline{\overline{v}}_2|$ and $|\frac{\mathrm{d}}{\mathrm{d}t}\overline{\overline{z}}_1| = |\frac{\mathrm{d}}{\mathrm{d}t}\overline{\overline{z}}_2|$, \textcolor{black}{implying that $\|\overline{\overline{v}}_2\|_{C_t^1} + \|\overline{\overline{z}}_2\|_{C_t^1}\leq \gamma$.}

\medskip

\noindent \textit{Step 1.3 (see Fig. \ref{fig:A} top-right).}
Now that we have $\overline{\overline{y}}_i$, we consider the sidewise (rightward) problem on $[0, \ell_i] \times [0, T]$ for the edges $i\in \{1, 2\}$ (see Remark \ref{rem:controllability_thm} \ref{subrem:sidewise}), where at $x=0$ the ``initial data'' is $\overline{\overline{y}}_i$, at $t=0$ the ``boundary condition'' prescribes the velocities as $v_i(x,0) = v_i^0(x)$ (thus using a part of the initial conditions of System \eqref{eq:syst_physical}), and at $t=T$ we set the artificial ``boundary condition'' $z_i(x,T) = \overline{q}_i(x)$ for any function $\overline{q}_i \in C^1([0, \ell_i]; \mathbb{R}^6)$.
Then, for any $\varepsilon_1>0$ small enough, there exists $\delta_1>0$ such that the rightward problem admits a unique solution $y_i \in C^1([0, \ell_i]\times[0, T]; \mathbb{R}^{12})$ with $\|y_i\|_{C_{x,t}^1}\leq \varepsilon_1$, provided that $\|\overline{\overline{y}}_i\|_{C_{t}^1} + \|v_i^0\|_{C_x^1} + \|\overline{q}_i\|_{C_x^1}\leq \delta_1$ for all $i \in \{1, 2\}$.

\begin{figure} \centering
\includegraphics[scale=0.55]{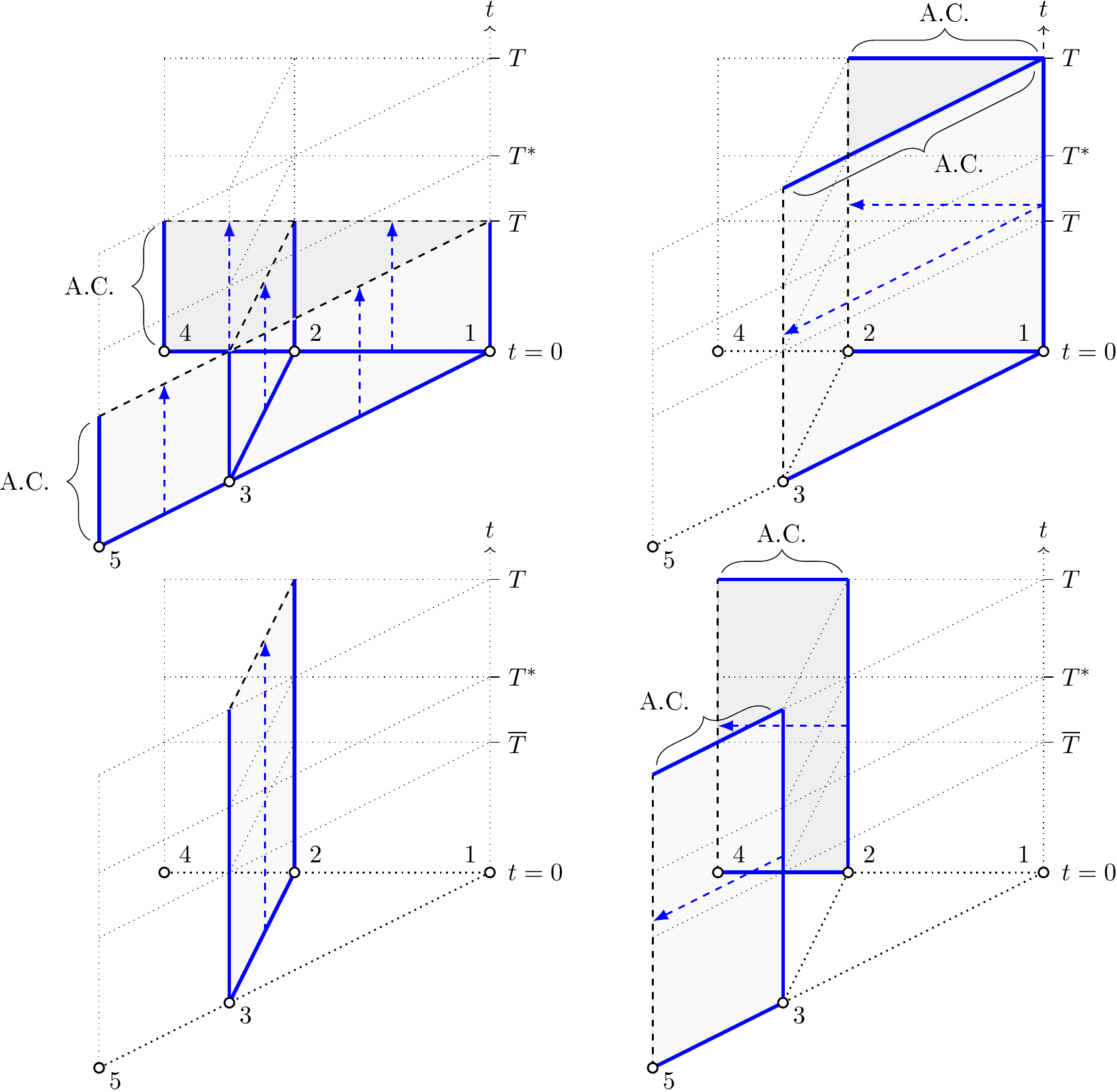}
\caption{Steps 1.1, 1.3, 1.4, 1.5 of the construction of the solution (top to bottom, left to right), where ``A.C.'' stands for ``artificial conditions''.}
\label{fig:A}
\end{figure}

\medskip

\noindent \textit{Step 1.4 (see Fig. \ref{fig:A} bottom-left).}
Using $y_1(\ell_1, \cdot)$, $y_2(\ell_2, \cdot)$ (from Step 1.3) as data, consider the forward problem on $[0, \ell_3]\times[0, T]$ for the edge $i=3$, with the initial conditions of \eqref{eq:syst_physical}, and, as boundary conditions at $x=0$ and $x=\ell_3$, the velocities prescribed as
\begin{linenomath}
\begin{align} \label{eq:pb3full_BC_0_l3}
v_3 (0, t) = \overline{R}_3(0)^\intercal \overline{R}_1(\ell_1) v_1(\ell_1, t), \qquad 
v_3 (\ell_3, t) = \overline{R}_3(\ell_3)^\intercal \overline{R}_2(\ell_2) v_2(\ell_2, t),
\end{align}
\end{linenomath}
so that the obtained solution $y_3$ together with $y_1, y_2$ (provided by Step 1.3) fulfill the continuity conditions \eqref{eq:IGEB_cont_velo} at the nodes $n \in \{2, 3\}$.
Then, for any $\varepsilon_2>0$ small enough, there exists $\delta_2>0$ such that this problem admits a unique solution $y_3 \in C^1([0, \ell_3]\times[0, T]; \mathbb{R}^{12})$ with $\|y_3\|_{C_{x,t}^1}\leq \varepsilon_2$, provided that $\|y_3^0\|_{c_x^1}+\|v_i(\ell_i, \cdot)\|_{C_t^1}+\|q_n\|_{C_t^1} \leq \delta_2$ for all $i \in \{1, 2\}$ and $n \in \{2, 3\}$.

\medskip

\noindent \textit{Step 1.5 (see Fig. \ref{fig:A} bottom-right).}
Finally, using $y_1(\ell_1, \cdot)$, $y_2(\ell_2, \cdot)$ (from Step 1.1) and $y_3(0, \cdot)$, $y_3(\ell_3, \cdot)$ (from Step 1.3) as data, consider the rightward problem on $[0, \ell_i]\times[0, T]$ for the edges $i \in \{4, 5\}$ similar to that of Step 1.3 except for the choice of the ``initial data'' at $x=0$, denoted by $\widetilde{y}_i$, that we define by
\begin{linenomath}
\begin{align}\label{eq:v45_x=0}
\widetilde{y}_4 &= \begin{bmatrix}
\overline{R}_4(0)^\intercal (\overline{R}_1v_1)(\ell_1, \cdot)\\
\overline{R}_4(0)^\intercal( (\overline{R}_1z_1)(\ell_1, \cdot) - (\overline{R}_3z_3)(0, \cdot)-q_2)
\end{bmatrix}\\
\label{eq:z45_x=0}
\widetilde{y}_5 &= \begin{bmatrix}
\overline{R}_5(0)^\intercal (\overline{R}_2 v_2)(\ell_2, \cdot)\\
\overline{R}_5(0)^\intercal ((\overline{R}_2 z_2)(\ell_2, \cdot) +(\overline{R}_3 z_3)(\ell_3, \cdot) - q_3)
\end{bmatrix}.
\end{align}
\end{linenomath}
Then, for any $\varepsilon_3>0$ small enough, there exists $\delta_3>0$ such that this problem admits a unique solution $y_i \in C^1([0, \ell_i]\times[0 , T]; \mathbb{R}^{12})$ with $\|y_i\|_{C_{x,t}^1}\leq \varepsilon_3$, provided that $\|v_i^0\|_{C_x^1}+ \|\overline{q}_i\|_{C_x^1} + \|q_n\|_{C_t^1} \leq \delta_3$ for all $i \in \{4,5\}$ and $n \in\{2, 3\}$, and $\|y_k(\ell_k, \cdot)\|_{C_t^1} + \|y_3(0, \cdot)\|_{C_t^1} \leq \delta_3$ for all $k\in \{1, 2, 3\}$.

Note that the $\widetilde{y}_i$ for $i \in \{4, 5\}$ have been chosen in such a way that the solutions $y_4, y_5$ together with $y_1, y_2, y_3$ (provided by Step 1.1 and Step 1.3), necessarily fulfill the transmission conditions \eqref{eq:IGEB_cont_velo}-\eqref{eq:IGEB_Kirchhoff} at the nodes $n \in \{2, 3\}$.

\medskip

\noindent It remains to prove that the solution $(y_i)_{i\in\mathcal{I}}$ constructed in Step 1 in fact also fulfills the initial conditions \eqref{eq:IGEB_ini_cond} of the overall network, by showing that $y_i$ coincides with $y_i^f$ on some domain including $[0, \ell_i]\times \{0\}$ for all $i \in \mathcal{I}$.

\medskip

\noindent \textit{Step 2.1 (see Fig. \ref{fig:A_ini} leftmost).}
First, consider the edges $i \in \{1, 2\}$. We will see that not only $y_i$ fulfills \eqref{eq:IGEB_ini_cond}, but one also has (see \eqref{eq:def_Lambdai_Ti})
\begin{linenomath}
\begin{align}
\label{eq:coinc12_BC}
y_i(\ell_i, t) = y_i^f(\ell_i, t), \quad &t \in [0, \max\{T_4, T_5\}], \, i \in \{1, 2\}.
\end{align}
\end{linenomath}
Let $i \in \{1, 2\}$, and let $\mathbf{t}_i \in C^1([0, \ell_i])$ be the function with derivative $\frac{\mathrm{d}}{\mathrm{d}x}\mathbf{t}_i(x) = \min_{1\leq k \leq 12} \frac{1}{\lambda_i^k(x)}$ in $[0, \ell_i]$, which is also equal to $-\Lambda_i(x)$ (see \eqref{eq:sign_eigval}-\eqref{eq:def_Lambdai_Ti}), and such that $\mathbf{t}_i(0) = T_i + \max \{T_4, T_5\}$. Then, $\mathbf{t}_i$ describes a curve in $[0, \ell_i]\times[0, T]$ that passes through $(0, T_i + \max\{T_4, T_5\})$ and we may also write
\begin{linenomath}
\begin{align*}
    \mathbf{t}_i(x) = T_i + \max \{T_4, T_5\} - \int_0^x \Lambda_i(s) ds.
\end{align*}
\end{linenomath}
The definition of $T_i$ in \eqref{eq:def_Lambdai_Ti} ensures that $[0, \ell_i]\times[0, \max\{T_4, T_5\}]$ is a subset of the domain $\mathcal{R}(i, \mathbf{t}_i)$ defined by
\begin{linenomath}
\begin{align}\label{eq:def_dom_calR}
\mathcal{R}(i, \mathbf{t}_i) := \{(x,t)\colon 0 \leq x \leq \ell_i, \ 0 \leq t \leq \mathbf{t}_i(x)\}.
\end{align}
\end{linenomath}
Both $y_i$ and $y_i^f$ are by definition solutions to the one-sided sidewise (rightward) problem with ``initial data'' $\overline{\overline{y}}_i$ at $x=0$ and boundary data $v_i^0$ at $t=0$.
The definition of $\mathbf{t}_i$ ensures that any characteristic curve\footnote{By characteristic curves passing by $(x_\circ,t_\circ)$, we mean the curves specified by the functions $\mathbf{t}_i^k$ with derivative $\frac{\mathrm{d}}{\mathrm{d}s} \mathbf{t}_i^k(s) = \lambda_i^k(s)^{-1}$ and such that $\mathbf{t}_i^k(x_\circ) = t_\circ$, for $k \in \{1, \ldots, 12\}$.
}
of this problem passing by $(x,t) \in \mathcal{R}(i, \mathbf{t}_i)$ is necessarily entering the domain $\mathcal{R}(i, \mathbf{t}_i)$ at $\{0\} \times [0, T_i+\max \{T_4, T_5\}]$ or at $[0, \ell_i] \times \{0\}$. Thus, by \cite[Section 1.7]{li2016book} the solution in $C^1(\mathcal{R}(i, \mathbf{t}_i); \mathbb{R}^{12})$ to this sidewise problem is unique, and $y_i \equiv y_i^f$ in $\mathcal{R}(i, \mathbf{t}_i)$.

\begin{figure}\centering
\includegraphics[scale=0.7]{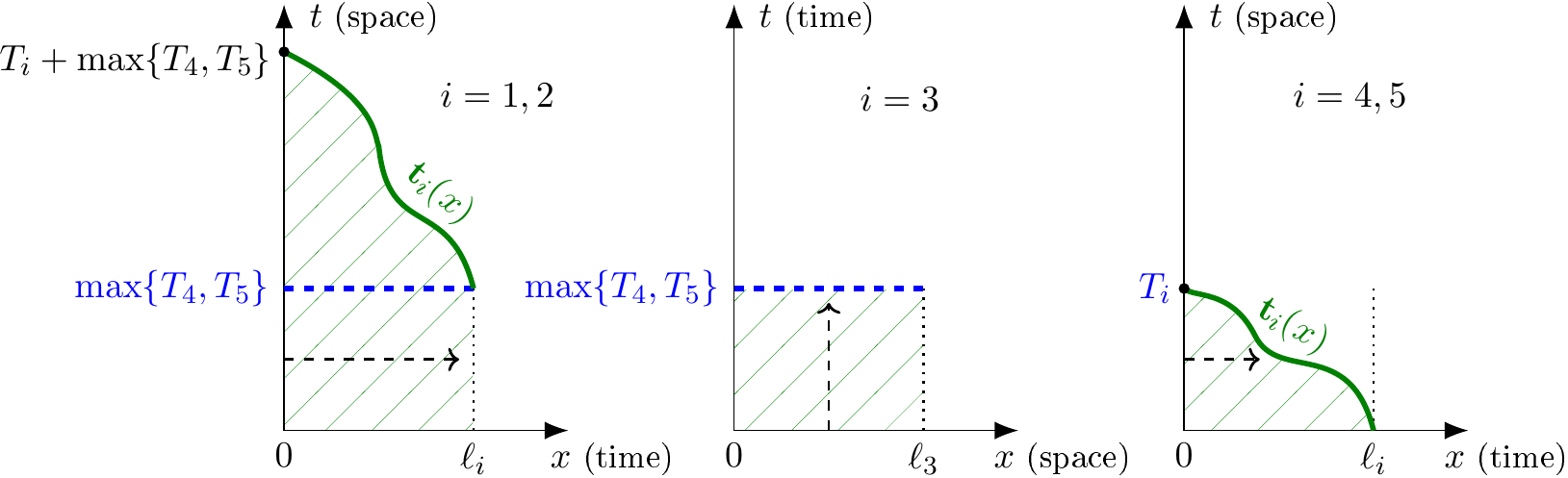}
\caption{Recovering the initial conditions: \textcolor{black}{meaning of the controllability time}.}
\label{fig:A_ini}
\end{figure}
\medskip

\noindent \textit{Step 2.2 (see Fig. \ref{fig:A_ini} center).}
Consider the edge $i=3$. We will show that not only $y_3$ fulfills \eqref{eq:IGEB_ini_cond}, but also
\begin{linenomath}
\begin{align} \label{eq:coinc3_BC}
y_3(0, t) = y_3^f(0, t), \quad y_3(\ell_3, t) = y_3^f(\ell_3, t), \quad t \in [0, \max \{T_4, T_5\}]
\end{align}
\end{linenomath}
holds. Indeed, $y_3$ and $y_3^f$ both solve the forward problem
\begin{linenomath}
\begin{subnumcases}{\label{eq:pb3}}
\label{eq:pb3_gov}
\partial_t \mathbf{y}_3 + A_i \partial_x \mathbf{y}_3 + \overline{B}_i \mathbf{y}_3 = \overline{g}_i(\cdot,\mathbf{y}_3) &\hspace{-0.45cm}$\text{in } (0, \ell_3)\times(0, \max \{T_4, T_5\})$\\
\label{eq:pb3_BC_0}
\mathbf{v}_3 (0, t) = \overline{R}_3(0)^\intercal (\overline{R}_1 v_1)(\ell_1, t) (t) &\hspace{-0.45cm}$t\in (0, \max \{T_4, T_5\})$\\
\label{eq:pb3_BC_l3}
\mathbf{v}_3 (\ell_3, t) = \overline{R}_3(\ell_3)^\intercal (\overline{R}_2 v_2)(\ell_2, t) (t) &\hspace{-0.45cm}$t\in (0, \max \{T_4, T_5\})$\\
\label{eq:pb3_ini}
\mathbf{y}_3(x,0) = y_3^0(x) &\hspace{-0.45cm}$x \in (0, \ell_i)$,
\end{subnumcases}
\end{linenomath}
which admits a unique solution in $C^1([0, \ell_3]\times[0, \max\{T_4, T_5\}];\mathbb{R}^{12})$.
In fact, $y_3$ fulfills \eqref{eq:pb3} by definition (see Step 3); concerning $y_3^f$, it fulfills \eqref{eq:pb3_gov} and \eqref{eq:pb3_ini} by definition, while \eqref{eq:coinc12_BC} and \eqref{eq:pb3full_BC_0_l3} imply that $y_3^f$ fulfills \eqref{eq:pb3_BC_0} and \eqref{eq:pb3_BC_l3}.

\medskip

\noindent \textit{Step 2.3 (see Fig. \ref{fig:A_ini} rightmost).}
Finally, consider the edges $i\in\{4,5\}$.
Let $\mathbf{t}_i$ be the function defined just as in Step 2.1 except that $\mathbf{t}_i(0) = T_i$. In other words, 
\begin{linenomath}
\begin{align*}
\mathbf{t}_i(x) = T_i - \int_0^x \Lambda_i(s)ds
\end{align*}
\end{linenomath}
Here, the definition of $T_i$ in \eqref{eq:def_Lambdai_Ti} ensures that $\mathbf{t}_i(\ell_i) = 0$, and therefore the corresponding domain $\mathcal{R}(i, \mathbf{t}_i)$ defined by \eqref{eq:def_dom_calR}, contains $[0, \ell_i] \times \{0\}$.
Both $y_i$ and $y_i^f$ fulfill the following one-sided rightward problem with unknown $\mathbf{y}_i$:
\begin{linenomath}
\begin{subnumcases}{\label{eq:pb45}}
\label{eq:pb45_gov}
\partial_x \mathbf{y}_i + A_i^{-1} \partial_t \mathbf{y}_i + A_i^{-1} \overline{B}_i \mathbf{y}_i = (A_i^{-1} \overline{g}_i)(\cdot, y_i)  & $\text{in } \mathcal{R}(i, \mathbf{t}_i)$\\
\label{eq:pb45_BC}
\mathbf{v}_i(x,0) = v_i^0(x) & $x \in (0, \ell_i)$\\
\label{eq:pb45_ini}
\mathbf{y}_i (0, t) = \widetilde{y}_i(t) & $t\in (0, T_i)$,
\end{subnumcases}
\end{linenomath}
where $\widetilde{y}_i$ is defined by \eqref{eq:v45_x=0}-\eqref{eq:z45_x=0}.
Indeed, while it is clear that $y_i$ fulfills \eqref{eq:pb45} and $y_i^f$ fulfills \eqref{eq:pb45_gov}-\eqref{eq:pb45_BC} by definition, one also obtains, using \eqref{eq:coinc12_BC}, \eqref{eq:coinc3_BC} and the fact that $y_i^f$ satisfies the transmission conditions \eqref{eq:IGEB_cont_velo}-\eqref{eq:IGEB_Kirchhoff}, that $y_i^f$ also fulfills \eqref{eq:pb45_ini}.
The definition of $\mathbf{t}_i$ ensures that any characteristic curve of \eqref{eq:pb45} passing through $(x,t) \in \mathcal{R}(i, \mathbf{t}_i)$ is necessarily entering this domain at $\{0\} \times [0,T_i]$ or at $[0, \ell_i] \times \{0\}$.
Hence, similarly to Step 2.1, one can apply \cite[Section 1.7]{li2016book} to obtain that the solution in $C^1(\mathcal{R}(i, \mathbf{t}_i); \mathbb{R}^{12})$ to \eqref{eq:pb45} is unique.

\medskip

\noindent \textit{Step 3.}
Finally, we choose $q_i$ defined by $q_i(t) = z_i(\ell_i, t)$ for all $t \in [0, T], i \in \{4, 5\}$. In view of the uniqueness of the solution to \eqref{eq:syst_physical}, $q_4, q_5$ are controls satisfying the desired conditions of Theorem \ref{th:controllability}.
\end{proof}

\section{Relationship between the GEB and IGEB networks}
\label{sec:invert_transfo}

As in Section \ref{sec:exist}, we now consider a general network, and seek to prove Theorem \ref{thm:solGEB}. To do so, in Lemma \ref{lem:invert_transfo} below, we start by inverting, on some specific spaces, the transformation $\mathcal{T}$ defined in \eqref{eq:transfo} that relates the states of \eqref{eq:GEB_netw} and \eqref{eq:syst_physical}.
Henceforth, for any functions $(u_i)_{i \in \mathcal{I}}$ such that $u_i \colon [0, \ell_i]\times[0, T]\rightarrow \mathbb{R}^{12}$, we use the notation $u_i=(u_{i,1}^\intercal, \ldots, u_{i, 4}^\intercal)^\intercal$, where $u_{i,k} \colon [0, \ell_i]\times[0, T]\rightarrow \mathbb{R}^3$ for all $k \in \{1, \ldots, 4\}$.
Let us define the spaces
\begin{linenomath}
\begin{align*}
E_1 &= \big\{(\mathbf{p}_i, \mathbf{R}_i)_{i \in \mathcal{I}} \in {\textstyle \prod_{i=1}^N} C^2\left([0, \ell_i]\times[0, T]; \mathbb{R}^3 \times \mathrm{SO}(3)\right) \colon \eqref{eq:GEB_condNSv_p_R}, \eqref{eq:GEB_IC_0ord} \text{ hold} \big\}\\
E_2 &= \big\{(y_i)_{i \in \mathcal{I}} \in {\textstyle \prod_{i=1}^N} C^1\left([0, \ell_i]\times[0, T] ; \mathbb{R}^{12}\right) \colon u_i := \mathrm{diag}(\mathbf{I}_6, \mathbf{C}_i) y_i \text{ satisfies} \\
&\qquad \text{\eqref{eq:compat_last6eq}-\eqref{eq:compat_ini}-\eqref{eq:compat_nod}} \big\},
\end{align*}
\end{linenomath}
where \eqref{eq:compat_last6eq}-\eqref{eq:compat_ini}-\eqref{eq:compat_nod} are the following conditions:
\begingroup
\setlength{\tabcolsep}{1pt} % Default value: 6pt
\renewcommand{\arraystretch}{0.75} % Default value: 1
\begin{linenomath}
\begin{align}
\label{eq:compat_last6eq}
&\begin{aligned}
&\text{for all }i \in \mathcal{I}, \text{ in }(0, \ell_i)\times (0, T)\\
&\partial_t \begin{bmatrix}
u_{i,3} \\ u_{i,4}
\end{bmatrix} - \partial_x \begin{bmatrix}
u_{i,1} \\ u_{i,2}
\end{bmatrix} - \begin{bmatrix}
\widehat{\Upsilon}_c^i & \widehat{e}_1 \\
\mathbf{0}_3 & \widehat{\Upsilon}_c^i
\end{bmatrix}\begin{bmatrix}
u_{i,1} \\ u_{i,2}
\end{bmatrix}  = 
\begin{bmatrix}
\widehat{u}_{i,2} & \widehat{u}_{i,1}\\
\mathbf{0}_3 & \widehat{u}_{i,2}
\end{bmatrix} \begin{bmatrix}
u_{i,3} \\ u_{i,4}
\end{bmatrix},
\end{aligned}\\
\label{eq:compat_ini}
&\begin{aligned}
\text{for all }i \in \mathcal{I}, \text{ in } (0, \ell_i), \quad \tfrac{\mathrm{d}}{\mathrm{d}x}\mathbf{p}_i^0(\cdot) &= \mathbf{R}_i^0(\cdot) (u_{i,3}(\cdot, 0) + e_1),\\
\tfrac{\mathrm{d}}{\mathrm{d}x} \mathbf{R}_i^0(\cdot) &= \mathbf{R}_i^0(\cdot)(\widehat{u}_{i,4}(\cdot, 0) + \widehat{\Upsilon}_c^i(\cdot)),
\end{aligned}\\
\label{eq:compat_nod}
&\begin{aligned}
\text{for all } n \in \mathcal{N}_S^D, \text{ in } (0, T),\quad \tfrac{\mathrm{d}}{\mathrm{d}t} f_n^\mathbf{p} (\cdot) &= f_n^\mathbf{R}(\cdot) \widehat{u}_{{i^n},1}(\mathbf{x}_{i^n}^n, \cdot),\\
\tfrac{\mathrm{d}}{\mathrm{d}t} f_n^\mathbf{R} (\cdot) &= f_n^\mathbf{R}(\cdot) \widehat{u}_{{i^n},2}(\mathbf{x}_{i^n}^n, \cdot).
\end{aligned}
\end{align}
\end{linenomath}
\endgroup 
The following result then holds.

\begin{lemma} \label{lem:invert_transfo}
Assume that $(\mathbf{p}_i^0, \mathbf{R}_i^0, f_n^\mathbf{p}, f_n^\mathbf{R})$ are of regularity \eqref{eq:reg_Idata_GEB} and \eqref{eq:reg_Ndata_GEB_D}, and fulfill \eqref{eq:compat_GEB_-1}.
Then, the transformation $\mathcal{T}\colon E_1 \rightarrow E_2$, defined in \eqref{eq:transfo}, is bijective.
\end{lemma}

\begin{proof}[Proof of Lemma \ref{lem:invert_transfo}]
One can easily verify that $(\mathcal{T}_i(\mathbf{p}_i, \mathbf{R}_i))_{i \in \mathcal{I}}$ belongs to $E_2$ for any given $(\mathbf{p}_i, \mathbf{R}_i)_{i\in\mathcal{I}} \in E_1$, and $\mathcal{T}$ is thus well defined.

Let $(y_i)_{i\in\mathcal{I}} \in E_2$. We will now show that, there exists a unique $(\mathbf{p}_i, \mathbf{R}_i)_{i \in \mathcal{I}}$ such that $\mathcal{T}((\mathbf{p}_i, \mathbf{R}_i)_{i\in\mathcal{I}}) = (y_i)_{i\in\mathcal{I}}$.
Consider $(u_i)_{i \in \mathcal{I}}$ defined by $u_i := \mathrm{diag}(\mathbf{I}_6, \mathbf{C}_i) y_i$.
Let $i \in \mathcal{I}$, and let $n$ be the index of any node such that $i \in \mathcal{I}^n$.

\medskip

\noindent There exists a unique solution $\mathbf{R}_i \in C^2([0, \ell_i]\times[0, T]; \mathrm{SO}(3))$ to 
\begin{linenomath}
\begin{subnumcases}{\label{eq:overdetR}}
\label{eq:overdetR_govt}
\partial_t \mathbf{R}_i = \mathbf{R}_i \widehat{u}_{i,2}
& $\text{in }(0, \ell_i) \times (0,T)$\\
\label{eq:overdetR_govx}
\partial_x \mathbf{R}_i = \mathbf{R}_i (\widehat{u}_{i,4} + \widehat{\Upsilon}_c^i) & $\text{in }(0, \ell_i) \times (0,T)$\\
\label{eq:overdetR_IBC}
\mathbf{R}_i(\mathbf{x}_i^n, 0) = \mathbf{R}_i^0(\mathbf{x}_i^n).
\end{subnumcases}
\end{linenomath}
To prove this, a possible way is to first rewrite \eqref{eq:overdetR}, whose state has values in $\mathrm{SO}(3)$, as a system with a $\mathbb{R}^4$-valued state (using \cite[Lem. 4.1]{RL2019}) via a parametrization of rotation matrices by quaternions \cite{chou1992}, and then use \eqref{eq:compat_last6eq} (last three equations) as compatibility conditions in order to deduce that the obtained system is well-posed (using \cite[Lem. 4.3]{RL2019}); this procedure is detailed in \cite[Section 4]{RL2019}.

\medskip

\noindent Having found $\mathbf{R}_i$, consider the following system
\begin{linenomath}
\begin{subnumcases}{\label{eq:overdetp}}
\label{eq:overdetp_govt}
\partial_t \mathbf{p}_i = \mathbf{R}_i u_{i,1}
& $\text{in }(0, \ell_i) \times (0,T)$\\
\label{eq:overdetp_govx}
\partial_x \mathbf{p}_i = \mathbf{R}_i (u_{i,3} + e_1) & $\text{in }(0, \ell_i) \times (0,T)$\\
\label{eq:overdetp_IBC}
\mathbf{p}_i(\mathbf{x}_i^n, 0) = \mathbf{p}_i^0(\mathbf{x}_i^n).
\end{subnumcases}
\end{linenomath}
Note that \eqref{eq:overdetp_govt} is equivalent to $\mathbf{p}_i(x,t) = \mathbf{p}_i(x,0) + \int_0^t (\mathbf{R}_i u_{i,1})(x, \tau)d\tau$. 
Without loss of generality, assume that $\mathbf{x}_i^n = 0$ (in the alternative case, the end of the proof is the same with each integral $+ \int_{0}^x$ below replaced by $- \int_x^{\ell_n}$).
By \eqref{eq:compat_ini} (first equation) and \eqref{eq:overdetp_IBC}, in the above expression for $\mathbf{p}_i(x,t)$, one may express the first term as $\mathbf{p}_i(x,0) = \mathbf{p}_i^0(\mathbf{x}_i^n) + \int_{\mathbf{x}_i^n}^x (\mathbf{R}_i^0(u_{i,3}^0+e_1))(s)ds$. Also, for any $x \in [0, \ell_i]$ and any $\tau \in [0, t]$ the integrand in the second term may be expressed as $(\mathbf{R}_i u_{i,1})(x, \tau) = (\mathbf{R}_i u_{i,1})(\mathbf{x}_i^n, \tau) + \int_{\mathbf{x}_i^n}^x (\mathbf{R}_i u_{i,1})(s, \tau) ds$. Hence, \eqref{eq:overdetp_govt} and \eqref{eq:overdetp_IBC} are equivalent to
\begin{linenomath}
\begin{align} \label{eq:pi_candidate}
\begin{aligned}
\mathbf{p}_i(x,t) &= \mathbf{p}_i^0(\mathbf{x}_i^n) + \int_0^t (\mathbf{R}_i u_{i,1})(\mathbf{x}_i^n, \tau)d\tau\\
&+ \int_{\mathbf{x}_i^n}^x (\mathbf{R}_i^0(u_{i,3}^0+e_1))(s)ds + \int_0^t \int_{\mathbf{x}_i^n}^x \partial_x (\mathbf{R}_i u_{i,1})(s, \tau)d\tau ds.
\end{aligned}
\end{align}
\end{linenomath}
On the other hand, we know that \eqref{eq:pi_candidate} fulfills  $\partial_t \mathbf{p}_i(\mathbf{x}_i^n, \cdot) = (\mathbf{R}_i u_{i,1})(\mathbf{x}_i^n, \cdot)$, while by \eqref{eq:compat_last6eq} (first three equations), one has $\partial_x (\mathbf{R}_i u_{i,1}) = \partial_t(\mathbf{R}_i(u_{i,3}+e_1))$. The latter two facts, together with \eqref{eq:compat_ini} (second equation), permit us to deduce that \eqref{eq:pi_candidate} also writes as $\mathbf{p}_i(x, t) = \mathbf{p}(\mathbf{x}_i^n, t) + \int_{\mathbf{x}_i^n}^x (\mathbf{R}_i(u_{i,3} + e_1))(t,s)ds$.
Thus, \eqref{eq:pi_candidate} is the unique solution to \eqref{eq:overdetp}.

Finally, note that, because of \eqref{eq:compat_ini}, requiring \eqref{eq:overdetR_IBC} and \eqref{eq:overdetp_IBC} is equivalent to imposing the initial conditions \eqref{eq:GEB_IC_0ord}. Moreover, in the case of $n \in \mathcal{N}_S^D$, due to \eqref{eq:compat_GEB_-1} and \eqref{eq:compat_nod}, requiring \eqref{eq:overdetR_IBC} and \eqref{eq:overdetp_IBC} is equivalent to imposing the nodal conditions \eqref{eq:GEB_condNSv_p_R}. This concludes the proof of Lemma \ref{lem:invert_transfo}.
\end{proof}

We now have the tools to prove Theorem \ref{thm:solGEB}.

\begin{proof}[Proof of Theorem \ref{thm:solGEB}]
We divide the proof in seven steps. Let $(y_i)_{i \in \mathcal{I}}$ be as in Theorem \ref{thm:solGEB}, and let $(u_i)_{i \in\mathcal{I}}$ be defined by $u_i = \mathrm{diag}(\mathbf{I}_6, \mathbf{C}_i) y_i$.

\medskip

\noindent \textit{Step 1: inverting the transformation.}
Since the last six equations in \eqref{eq:IGEB_gov} hold for $(y_i)_{i\in\mathcal{I}}$, we know that \eqref{eq:compat_last6eq} is fulfilled. On the other hand, the last six equations of the initial conditions \eqref{eq:IGEB_ini_cond} with initial data \eqref{eq:rel_inidata} yield \eqref{eq:compat_ini}. Finally, the definition of the boundary data \eqref{eq:def_qn_D}, together with the nodal conditions \eqref{eq:IGEB_condNSv} on velocities, yield \eqref{eq:compat_nod}. Hence, $(y_i)_{i\in\mathcal{I}} \in E_2$, and by Lemma \ref{lem:invert_transfo} there exists a unique $(\mathbf{p}_i, \mathbf{R}_i)_{i \in \mathcal{I}} \in E_1$ such that
\begin{linenomath} 
\begin{align} \label{eq:transfo_inverted}
y_i = \mathcal{T}_i(\mathbf{p}_i, \mathbf{R}_i), \quad \text{for all }i \in \mathcal{I}.
\end{align}
\end{linenomath}

Now, we want to check that this ``candidate'' $(\mathbf{p}_i, \mathbf{R}_i)_{i \in \mathcal{I}}$, satisfies the rest of system \eqref{eq:GEB_netw}.

\medskip

\noindent \textit{Step 2: governing equations.}
Using \eqref{eq:transfo_inverted} and the first six governing equations in \eqref{eq:IGEB_gov}, one can deduce that $(\mathbf{p}_i, \mathbf{R}_i)_{i \in \mathcal{I}}$ satisfies the governing system \eqref{eq:GEB_gov} after some computations.

\medskip

\noindent \textit{Step 3: conditions at simple nodes.}
For $n \in \mathcal{N}_S^N$, from \eqref{eq:transfo_inverted} together with the nodal conditions \eqref{eq:IGEB_condNSz} on forces and moments and the definition of $f_n$ (see \eqref{eq:def_fn}), one can directly deduce that the nodal conditions \eqref{eq:GEB_condNSz} hold.

For $n \in \mathcal{N}_S^D$, from \eqref{eq:transfo_inverted} together with the nodal conditions \eqref{eq:IGEB_condNSv} on velocities and initial conditions \eqref{eq:GEB_IC_0ord}, we deduce that $(\mathbf{p}_{i^n}(\mathbf{x}_{i^n}^n, \cdot), \mathbf{R}_{i^n}(\mathbf{x}_{i^n}^n, \cdot))$ satisfies
\begin{linenomath}
\begin{align} \label{eq:nodPDE}
\left\{ \begin{aligned}
&\frac{\mathrm{d}\beta}{\mathrm{d}t}(t) = \beta(t)  \widehat{q_n^W}(t), \ \ \frac{\mathrm{d}\alpha}{\mathrm{d}t}(t) = \beta(t) q_n^V(t) && \ \text{ in }(0, T)\\
&(\alpha, \beta)(0) = (\mathbf{p}_{i^n}^0, \mathbf{R}_{i^n}^0)(\mathbf{x}_{i^n}^n),
\end{aligned} \right.
\end{align}
\end{linenomath}
of unknown state $(\alpha, \beta)$, where we denote $q_n = ((q_n^V)^\intercal, (q_n^W)^\intercal)^\intercal$ with $q_n^V, q_n^W \in C^1([0, T]; \mathbb{R}^3)$.
Due to \eqref{eq:def_qn_D} and \eqref{eq:compat_GEB_-1}, $(f_n^\mathbf{p}, f_n^\mathbf{R})$ also satisfies \eqref{eq:nodPDE}.
One may see that \eqref{eq:nodPDE} admits a unique solution in $C^2([0, T]; \mathbb{R}^{3}\times \mathrm{SO}(3))$. Indeed, as in the proof of Lemma \ref{lem:invert_transfo}, one may replace \eqref{eq:nodPDE} (first equation) by an equivalent equation whose unknown state is the quaternion \cite{chou1992} parametrizing the rotation matrix $\beta = \beta(t)$ (see \cite[Section 4]{RL2019} for more detail). Having then only vector valued unknowns, one can use the classical ODE theory. Thus, $(\mathbf{p}_{i^n},\mathbf{R}_{i^n})(\mathbf{x}_{i^n}^n, \cdot) \equiv (f_n^\mathbf{p}, f_n^\mathbf{R})$, and the nodal conditions \eqref{eq:GEB_condNSv_p_R} hold.

\medskip

\noindent \textit{Step 4: remaining initial conditions.}
One recovers the initial conditions \eqref{eq:GEB_IC_1ord} directly from the first six equations in \eqref{eq:IGEB_ini_cond} and the definition of $y_i^0$ \eqref{eq:rel_inidata}, together with \eqref{eq:transfo_inverted}.

\medskip

\noindent \textit{Step 5: rigid joint condition.}
In order to show that $(\mathbf{p}_i, \mathbf{R}_i)_{i \in \mathcal{I}}$ fulfills the transmission conditions of \eqref{eq:GEB_netw}, we start with the rigid joint condition. Let $n \in \mathcal{N}_M$.  For all $i \in \mathcal{I}^n$, let us define $\Lambda_i \in  C^1([0, T]; \mathbb{R}^{3 \times 3})$ by $\Lambda_i(t) = (R_i \mathbf{R}_i^\intercal)(\mathbf{x}_i^n, t)$. By the continuity condition \eqref{eq:IGEB_cont_velo} (last three equations),
\begin{linenomath}
\begin{align} \label{cont_derivative_angle}
\textstyle
\left(\frac{\mathrm{d}}{\mathrm{d}t} \Lambda_i \right) \Lambda_i^\intercal  = \left(\frac{\mathrm{d}}{\mathrm{d}t} \Lambda_{i^n} \right)\Lambda_{i^n}^\intercal, \quad \text{in }(0, T), \text{ for all }i\in\mathcal{I}^n.
\end{align}
\end{linenomath}
Let $F_n := \left(\frac{\mathrm{d}}{\mathrm{d}t} \Lambda_{i^n} \right)\Lambda_{i^n}^\intercal$ and $a_n := (R_{i^n}{\mathbf{R}_{i^n}^0}^\intercal)(\mathbf{x}_{i^n}^n)$. By \eqref{cont_derivative_angle}, \eqref{eq:compat_GEB_transmi} (second equation) and the fact that \eqref{eq:GEB_IC_0ord} holds (by Step 1), for all $i \in \mathcal{I}^n$, $\Lambda_i$ fulfills
\begin{linenomath}
\begin{align*}
\begin{dcases}
\frac{\mathrm{d}}{\mathrm{d}t} \Lambda_i(t) = F_n(t) \Lambda_i(t) & \text{for all }t \in (0, T)\\
\Lambda_i(0) = a_n,
\end{dcases}
\end{align*}
\end{linenomath}
which admits a unique $C^1([0, T]; \mathbb{R}^{3 \times 3})$ solution (see \cite[Sec. 2.1 and Th. 4.1.1 or Coro. 2.4.4]{vrabie2004}, for instance). Hence, $\Lambda_i \equiv \Lambda_j$ for all $i,j\in\mathcal{I}^n$, and the rigid joint condition \eqref{eq:GEB_rigid_angles} holds.

As \eqref{eq:GEB_rigid_angles} holds, we can now deduce the transmission conditions of \eqref{eq:GEB_netw} that remain.

\medskip

\noindent \textit{Step 6: continuity of the displacement.}
Let $n \in \mathcal{N}_M$.
By \eqref{eq:transfo_inverted} together with the rigid joint condition \eqref{eq:GEB_rigid_angles} and the continuity condition \eqref{eq:IGEB_cont_velo} (first three equations), one deduces that 
\begin{linenomath}
\begin{align*} %\label{eq:cont_derivative_pi}
\partial_t \mathbf{p}_i(\mathbf{x}_i^n, t)  = \partial_t \mathbf{p}_{i^n}(\mathbf{x}_{i^n}^n, t), \quad \text{in }(0, T), \text{ for all } i\in\mathcal{I}^n.
\end{align*}
\end{linenomath}
Using additionally \eqref{eq:GEB_IC_0ord} with \eqref{eq:compat_GEB_transmi} (first equation), we deduce that for all $i \in \mathcal{I}^n$, the function $\mathbf{p}_i(\mathbf{x}_i^n, \cdot)$ fulfills the problem
\begin{linenomath}
\begin{align} \label{eq:contdipl_ODE}
\begin{dcases}
\partial_t \mathbf{p}_i(\mathbf{x}_i^n, t) = h_n(t) & \text{for all }t \in (0, T)\\
\mathbf{p}_i(\mathbf{x}_i^n, 0) = \alpha_n,
\end{dcases}
\end{align}
\end{linenomath}
where we denote $h_n := \partial_t \mathbf{p}_{i^n}(\mathbf{x}_{i^n}^n, \cdot)$ and $\alpha_n := \mathbf{p}_{i^n}^0(\mathbf{x}_{i^n}^n)$. Since the $C^1([0, T];\mathbb{R}^3)$ solution to \eqref{eq:contdipl_ODE} is unique, we conclude that \eqref{eq:GEB_continuity_pi} holds.

\medskip

\noindent \textit{Step 7: Kirchhoff condition.}
One recovers the Kirchhoff condition \eqref{eq:GEB_Kirchhoff} from the rigid joint assumption \eqref{eq:GEB_rigid_angles} together with \eqref{eq:IGEB_Kirchhoff} and \eqref{eq:transfo_inverted}.

To finish, the uniqueness of the solution to \eqref{eq:GEB_netw} is a consequence of the uniqueness of the solution to \eqref{eq:syst_physical} (Theorem \ref{th:existence}) and of the bijectivity of the transformation $\mathcal{T}$ (Lemma \ref{lem:invert_transfo}). This concludes the proof.
\end{proof}

\section{Concluding remarks and outlook}

\label{sec:conclusion}

In this article, we have studied networks, possibly with cycles, of geometrically exact beams. Notably, we considered the representations of such beams in terms of either displacements and rotations expressed in a fixed coordinate system (GEB model), or velocities and internal forces/moments expressed in a moving coordinate system attached to the beam (IGEB model), reflecting on the advantages and drawbacks of these two points of view, and the relationship between them. 
For these beam networks, we addressed the problem of local exact controllability of nodal profiles in the special case of a network containing one cycle: the A-shaped network depicted in Fig. \ref{subfig:AshapedNetwork}.

The fact that one has the possibility of expressing the beam model as a first-order semilinear hyperbolic system -- the IGEB model -- while keeping track of the link with the GEB model, permits us to give a proof of nodal profile controllability in line with works done on other one-dimensional hyperbolic systems -- e.g., wave equation, Saint-Venant equations, Euler equations \cite{gu2011, gu2013, gugat10, li2010nodal, li2016book, kw2011, kw2014, YWang2019partialNP, Zhuang2018}. Namely, we used the existence and uniqueness theory of semi-global classical solutions to the network system, combined with a constructive method as in \cite{Zhuang2018} to obtain adequate controls.

\medskip

\noindent \textbf{Local nature of the results.}
{\color{black} 
Let us give some comments about the local nature of the nodal profile controllability result, Theorem \ref{th:controllability}. This theorem notably implies that even though there might be large displacements and rotations of the beam -- due to the use of a geometrically exact (thus nonlinear) beam model --, we apply controls that subsequently keep these motions small. As noted in Remark \ref{rem:controllability_thm} \ref{subrem:global}, Theorem \ref{th:controllability} focuses on the small data scenario and could possibly be a preliminary step in view of obtaining a global result.
%It is first necessary to look at the small data scenario to build on top of that and thus such a result is a first step. 
%(or perturbations from any nonlinear equilibrium small).

Besides, in the proof of Theorem \ref{th:controllability}, some ``degree of freedom'' has not been used, as we rely on an existence and uniqueness result which has been established for general one-dimensional first-order quasilinear hyperbolic systems. Since we are considering a very specific model -- the IGEB model -- it would be interesting to establish an appropriate well-posedness result and keep track of the bounds on the initial and boundary data to obtain more quantitative information.

On another hand, as explained in the introduction, the GEB and IGEB models are valid as long as the strains $s_i(x,t)$ are small enough. The latter being proportional to the the internal forces and moments $z_i(x,t)$ (more precisely, they are given by $s_i = \mathbf{C}_i z_i$ where we recall that $\mathbf{C}_i(x)$ denotes the flexibility matrix), comparing this assumption to the smallness of the internal forces and moments required in Theorem \ref{th:controllability} would also be of interest.
}

\medskip

\begin{figure}
    \centering
    \includegraphics[scale = 0.8]{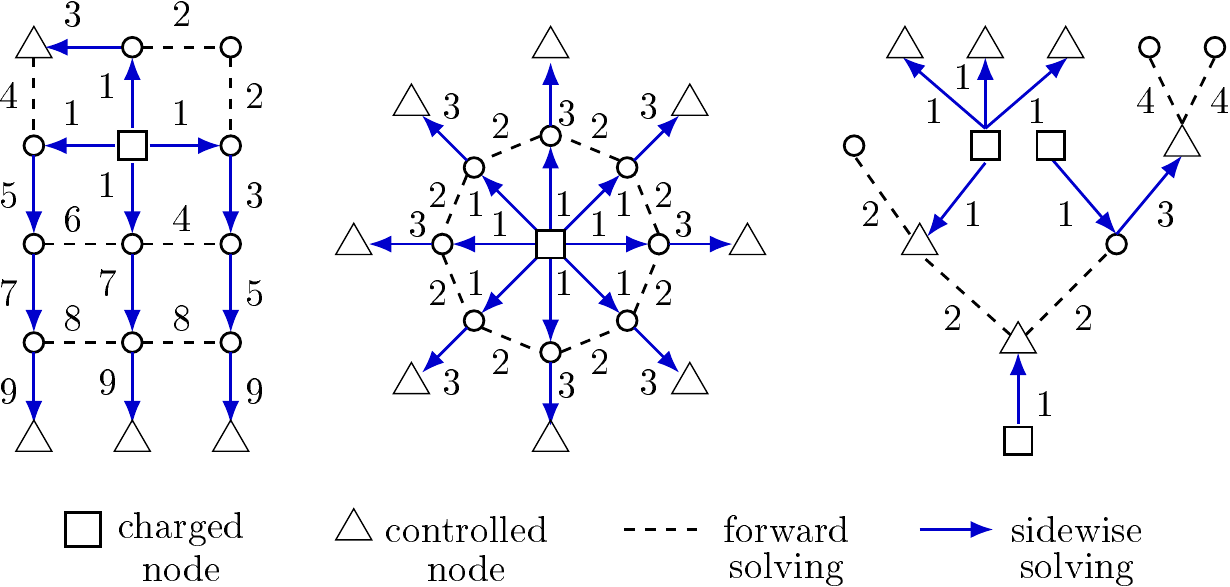}
    \caption{Other networks for which local exact controllability of nodal profiles is achievable by following Algorithm \ref{algo:control} (the numbers refer to the variable \textsf{step}).}
    \label{fig:otherControllableNet}
\end{figure}

\noindent \textbf{More general networks.} The A-shaped network is an illustrative example where the controllability of nodal profiles is achievable for a network with a cycle, but let us stress that similar arguments to those used in Section \ref{sec:controllability} apply for various other networks, and with controls at different locations.

\begin{algorithm}
\DontPrintSemicolon

\Input{%
$\mathcal{I}, \, \mathcal{N}, \, \mathcal{N}_S, \, k_n$\tcp*{edges, nodes, simple nodes, degrees}
\hspace{1.2cm}$\mathcal{I}^n$ for all $n\in \mathcal{N}$\tcp*{edges incident to the node $n$}
\hspace{1.2cm}$\mathcal{N}^i$ for all $i \in \mathcal{I}$\tcp*{nodes at the tips of the edge $i$}
\hspace{1.2cm}$\mathcal{P}$, $\mathcal{C}$\tcp*{charged nodes, controlled nodes}
\hspace{1.2cm}$\mathcal{S}$\tcp*{edges on control paths}
}

\setcounter{AlgoLine}{0}
\ShowLn
$J$ $\leftarrow$ $[\, 0, \quad $for $n = 1 \ldots \#\mathcal{N}\, ]$; \ \lFor{all $n \in \mathcal{P}$}{($J(n)$ $\leftarrow$ $k_n - 1$);} 
\tcp*{amount $J(n)$ of data available at $n$ to solve sidewise}

\ShowLn
$\mathcal{F}$ $\leftarrow$ $\emptyset$;\tcp*{solved edges}
 
\ShowLn
\textsf{step} $\leftarrow$ $1$; \ \textsf{moved} $\leftarrow$ \textsf{false};\tcp*{to count the steps}

\ShowLn
$\mathcal{M}$  $\leftarrow$ $\mathcal{P}\cup \{n \in \mathcal{N} \colon\text{$n$ not incident with any edge in $\mathcal{S}$}\};$\;
 
\ShowLn
\While(\hfill \tcp*[h]{while entire network not solved}){$\mathcal{F} \neq \mathcal{I}$}{
 
\ShowLn
\For(\hfill \tcp*[h]{Principle 1}){$m =\#\mathcal{M}, \ldots, 3, 2, 1$}{

\ShowLn
\For{all ${\mathcal{N}}^\dagger \subseteq \mathcal{M}$ such that $\#\mathcal{N}^\dagger = m$}{

\ShowLn
\If{there exists a connected subgraph with nodes ${\mathcal{N}}^\dagger$ and edges ${\mathcal{I}}^\dagger$, such that ${\mathcal{I}}^\dagger \cap ( \mathcal{S}\cup \mathcal{F}) = \emptyset$}{

\ShowLn
  solve forward problem for the network $({\mathcal{I}^\dagger}, {\mathcal{N}}^\dagger)$;\;
  
\ShowLn
\lFor{all $n \in {\mathcal{N}}^\dagger$}{
  ($J(n)$ $\leftarrow$ $J(n) + 1$);}
  
\ShowLn
$\mathcal{M}$ $\leftarrow$ $\mathcal{M} \cup {\mathcal{N}}^\dagger$; \, $\mathcal{F}$ $\leftarrow$ $\mathcal{F} \cup {\mathcal{I}}^\dagger$; \ \textsf{moved} $\leftarrow$ \textsf{true};\;
}}}

\ShowLn
\lIf{\upshape \textsf{moved} $=$ \textsf{true} }{(\textsf{step} $\leftarrow$ \textsf{step} $+ 1$; \ \textsf{moved} $\leftarrow$ \textsf{false});}  
   
\ShowLn
\For(\hfill \tcp*[h]{Principle 2}){all $n \in \mathcal{M}$}{

\ShowLn
\If(\hfill \tcp*[h]{if enough data at $n$}){$J(n) = k_n - 1$}{

\ShowLn
\For{all $i \in \mathcal{I}^n \cap (\mathcal{S} \setminus \mathcal{F}$)}{

\ShowLn
solve sidewise problem for the edge $i$ with ``initial \;
  conditions'' at the node $n$;\;

\ShowLn
$\mathcal{M}$ $\leftarrow$ $\mathcal{M} \cup \mathcal{N}^i$; \, $\mathcal{F}$ $\leftarrow$ $\mathcal{F}\cup \{i\}$;\;

\ShowLn
\lFor{all $m \in \mathcal{N}^i$}{
  ($J(m)$ $\leftarrow$ $J(m) + 1$);}
  } 
  
\ShowLn 
\textsf{moved} $\leftarrow$ \textsf{true};\;
}}

\ShowLn
\lIf{\upshape \textsf{moved} $=$ \textsf{true}}{(\textsf{step} $\leftarrow$ \textsf{step} $+ 1$; \ \textsf{moved} $\leftarrow$ \textsf{false});} 
}

\ShowLn  
Compute controls $q_n$ by evaluating the trace at nodes $n\in\mathcal{C}$;\;
 
\caption{Steps of controllability proof for other networks}
\label{algo:control}
\end{algorithm}

Let us introduce some more notation. For any given network, we denote by $\mathcal{P}$ and $\mathcal{C}$ the set of indexes of the \emph{charged nodes} and \emph{controlled nodes} (see Section \ref{sec:intro}), respectively.
Given a charged node $n\in\mathcal{P}$ and a controlled node $m\in\mathcal{C}$, a \emph{control path} between $n$ and $m$ \cite{YWang2019partialNP}, is any connected subgraph (of the current graph representing the beam network) forming a path graph\footnote{A path graph is an oriented graph without cycle such that two of its nodes are of degree $1$, and all other nodes have a degree equal to $2$.} whose nodes of degree $1$ are $n$ and $m$. In Fig. \ref{fig:otherControllableNet}, examples of control paths are highlighted by blue arrows.

\medskip

%\noindent On a \emph{tree-shaped} network -- hence without loop --, either for a complete combination of all state functions at one charged node (see in \cite{gu2011, gu2013, li2016book, kw2011, kw2014}), or for a part of the information of nodal profiles (see in \cite{YWang2019partialNP}), some conditions were proved to be \emph{sufficient} for the exact controllability of nodal profiles to be achieved.

%%%\noindent On a \emph{tree-shaped} network -- hence without loop --, either for profiles prescribing the information of the entire state of all incident edges (see in \cite{gu2011, gu2013, li2016book, kw2011, kw2014}), or part of the information (see in \cite{YWang2019partialNP}), some conditions were proved to be \emph{sufficient} for the exact controllability of nodal profiles to be achieved.

\noindent On a \emph{tree-shaped} network -- hence without loop --, some conditions were proved to be \emph{sufficient} for the exact controllability of nodal profiles to be achieved \cite{gu2011, gu2013, li2016book, kw2011, kw2014, YWang2019partialNP}.
In \cite{YWang2019partialNP} the authors are concerned with the wave equation and provide a controllability result for any given tree-shaped network with possibly several charged nodes. Moreover, in \cite{YWang2019partialNP}, at a charged node $n \in \mathcal{P}$, profiles may be prescribed for only some (rather than all) of the edges incident with $n$, and profiles may be prescribed for only part of the state (which would translate in the case of \eqref{eq:syst_physical} to prescribing the velocities only, or the internal forces and moments only, for example). This type of problem, which is then called \emph{partial nodal profile controllability}, is not considered here.

The nodal profile controllability has also been established for the Saint-Venant system \cite{Zhuang2021, Zhuang2018} for numerous networks \emph{with cycles}, of various shapes and with several charged nodes.

\medskip

\noindent It arises, from these works, a series of conditions on the number and location of the charged nodes, which are sufficient to achieve the respective controllability goals. We refer notably to \cite[Theorem 5.1]{YWang2019partialNP}, and to \cite[Sections 7 and 8]{Zhuang2021}.
%%%\textcolor{blue}{, where in \cite{YWang2019partialNP}, the given conditions involve the notion of \emph{freedom degree} $F(n)$ of a charged node $n \in \mathcal{P}$. It is defined as the difference between the number of given profiles at this node and the number of constraints imposed to the nodal profiles by the transmission conditions at this node. For instance, in the case of Theorem \ref{th:controllability}, the charged node $n=1$ has for freedom degree $F(n)=12$, since one gives $24$ nodal profiles, from which are deduced $12$ constraints imposed by \eqref{eq:IGEB_cont_velo}-\eqref{eq:IGEB_Kirchhoff}.}

In the case of the beam networks considered in this article, these conditions become (recall that $k_n$ is defined as the degree of the node $n$)
\begin{enumerate}
\item The total number of controlled nodes $\#\mathcal{C}$ is equal to $\sum_{n\in \mathcal{P}} k_n$.
\item For any $n \in \mathcal{P}$, there are $k_n$ controlled nodes connecting with it through control paths. These control paths have the charged node $n$ for sole common node. 
\item The control paths corresponding to different charged nodes do not have any common node.
\end{enumerate}
Let us stress again that we are restricting ourselves to the type of systems presented in Subsection \ref{subsec:network_systems}. Namely, if a multiple node is controlled, then the control is applied at the Kirchhoff condition, while if a simple node is controlled, then the control is applied at either the first six (velocities) or last six (internal forces and moments) components of the state $y_i$, and at any charged node $n \in\mathcal{P}$ profiles are prescribed for all incident beams $i \in \mathcal{I}^n$ and for the entire state $y_i$.

\medskip

\noindent Then, one may use the \emph{constructive} method as in Section \ref{sec:controllability}, by following the steps instructed by Algorithm \ref{algo:control}, for different networks; see Fig. \ref{fig:otherControllableNet}.
We can assert that this algorithm yields a proof of controllability for the networks defined in Fig. \ref{fig:otherControllableNet}, but not that it constitutes a proof for any given network.

In Algorithm \ref{algo:control}, edges belonging to control paths are solved according to the \emph{Principle 2} -- solving a sidewise problem as in the Steps 1.3 and 1.5 of the proof of Theorem \ref{th:controllability} -- while the other edges are solved according to the \emph{Principle 1} -- solving a forward problem similar to the Step 1.4 of the proof of Theorem \ref{th:controllability}.

As noted here and in the above cited works, the conditions given to obtain controllability of nodal profiles are only \emph{sufficient} to ensure the controllability result and the search for necessary and sufficient conditions is open.

%%%%%%%%%%%%%%%%%%%%%%%%%%%%%%%%%%%%%%%%%
%%%%%%%%%%%%%%% MAIN %%%%%%%%%%%%%%%%%%%%
%%%%%%%%%%%%%%%%%%%%%%%%%%%%%%%%%%%%%%%%%

\bibliographystyle{acm}
\bibliography{mybibfile}

\begin{thebibliography}{10}

\bibitem{Artola2020aero}
{\sc Artola, M., Goizueta, N., Wynn, A., and Palacios, R.}
\newblock Modal-based nonlinear estimation and control for highly flexible
  aeroelastic systems.
\newblock In {\em AIAA Scitech Forum\/} (2020).

\bibitem{Artola2019mpc}
{\sc Artola, M., Wynn, A., and Palacios, R.}
\newblock A nonlinear modal-based framework for low computational cost optimal
  control of 3{D} very flexible structures.
\newblock In {\em 18th European Control Conference\/} (2019), pp.~3836--3841.

\bibitem{Artola2021damping}
{\sc Artola, M., Wynn, A., and Palacios, R.}
\newblock Generalized {K}elvin-{V}oigt damping for geometrically nonlinear
  beams.
\newblock {\em AIAA Journal 59}, 1 (2021), 356--365.

\bibitem{BC2016}
{\sc Bastin, G., and Coron, J.-M.}
\newblock {\em {S}tability and {B}oundary {S}tabilization of 1-{D} {H}yperbolic
  {S}ystems}, vol.~88 of {\em Progr. Nonlinear Differential Equations Appl.}
\newblock Birkh\"{a}user/Springer, [Cham], 2016.

\bibitem{chen_serial_EBbeams}
{\sc Chen, G., Delfour, M.~C., Krall, A.~M., and Payre, G.}
\newblock Modeling, stabilization and control of serially connected beams.
\newblock {\em SIAM J. Control Optim. 25}, 3 (1987), 526--546.

\bibitem{chou1992}
{\sc Chou, J. C.~K.}
\newblock Quaternion kinematic and dynamic differential equations.
\newblock {\em IEEE Trans. Robot. Autom. 8}, 1 (1992), 53--64.

\bibitem{Macchelli2009book}
{\sc Duindam, V., Macchelli, A., Stramigioli, S., and Bruyninckx, H.}, Eds.
\newblock {\em {M}odeling and {C}ontrol of {C}omplex {P}hysical {S}ystems.
  {T}he port-{H}amiltonian approach}.
\newblock Springer-Verlag, Berlin, 2009.

\bibitem{evans2}
{\sc Evans, L.~C.}
\newblock {\em {P}artial {D}ifferential {E}quations}, second~ed., vol.~19 of
  {\em Graduate Studies in Mathematics}.
\newblock American Mathematical Society, Providence, RI, 2010.

\bibitem{grazioso2018robot}
{\sc Grazioso, S., Di~Gironimo, G., and Siciliano, B.}
\newblock A geometrically exact model for soft continuum robots: The finite
  element deformation space formulation.
\newblock {\em Soft robotics 6}, 6 (2019), 790--811.

\bibitem{gu2011}
{\sc Gu, Q., and Li, T.}
\newblock Exact boundary controllability of nodal profile for quasilinear
  hyperbolic systems in a tree-like network.
\newblock {\em Math. Methods Appl. Sci. 34\/} (2011), 911--928.

\bibitem{gu2013}
{\sc Gu, Q., and Li, T.}
\newblock Exact boundary controllability of nodal profile for unsteady flows on
  a tree-like network of open canals.
\newblock {\em J. Math. Pures Appl. 99}, 1 (2013), 86--105.

\bibitem{gugat10}
{\sc Gugat, M., Herty, M., and Schleper, V.}
\newblock Flow control in gas networks: exact controllability to a given
  demand.
\newblock {\em Math. Methods Appl. Sci. 34}, 7 (2011), 745--757.

\bibitem{gugatLeugering2003}
{\sc Gugat, M., and Leugering, G.}
\newblock Global boundary controllability of the de {S}t. {V}enant equations
  between steady states.
\newblock {\em Ann. Inst. H. Poincar\'{e} Anal. Non Lin\'{e}aire 20}, 1 (2003),
  1--11.

\bibitem{hodges1990}
{\sc Hodges, D.~H.}
\newblock A mixed variational formulation based on exact intrinsic equations
  for dynamics of moving beams.
\newblock {\em Int. J. Solids Struct. 26}, 11 (1990), 1253--1273.

\bibitem{hodges2003geometrically}
{\sc Hodges, D.~H.}
\newblock Geometrically exact, intrinsic theory for dynamics of curved and
  twisted anisotropic beams.
\newblock {\em AIAA Journal 41}, 6 (2003), 1131--1137.

\bibitem{Zwart2012bluebook}
{\sc Jacob, B., and Zwart, H.~J.}
\newblock {\em {L}inear {P}ort-{H}amiltonian {S}ystems on
  {I}nfinite-{D}imensional {S}paces}, vol.~223 of {\em Operator Theory:
  Advances and Applications}.
\newblock Birkh\"{a}user/Springer Basel AG, Basel, 2012.
\newblock Linear Operators and Linear Systems.

\bibitem{LLS}
{\sc Lagnese, J.~E., Leugering, G., and Schmidt, E. J. P.~G.}
\newblock {\em {M}odeling, {A}nalysis and {C}ontrol of {D}ynamic {E}lastic
  {M}ulti-{L}ink {S}tructures}.
\newblock Systems \& Control: Foundations \& Applications. Birkh\"{a}user
  Boston, Inc., Boston, MA, 1994.

\bibitem{li2010controllability}
{\sc Li, T.}
\newblock {\em {C}ontrollability and {O}bservability for {Q}uasilinear
  {H}yperbolic {S}ystems}, vol.~3 of {\em AIMS Ser. Appl. Math.}
\newblock Am. Inst. Math. Sci., Springfield, MO; Higher Education Press,
  Beijing, 2010.

\bibitem{li2010nodal}
{\sc Li, T.}
\newblock Exact boundary controllability of nodal profile for quasilinear
  hyperbolic systems.
\newblock {\em Math. Methods Appl. Sci. 33\/} (2010), 2101--2106.

\bibitem{LiJin2001_semiglob}
{\sc Li, T., and Jin, Y.}
\newblock Semi-global {$C^1$} solution to the mixed initial-boundary value
  problem for quasilinear hyperbolic systems.
\newblock {\em Chinese Ann. Math. Ser. B 22}, 3 (2001), 325--336.

\bibitem{LiRao2002_cam}
{\sc Li, T., and Rao, B.}
\newblock Local exact boundary controllability for a class of quasilinear
  hyperbolic systems.
\newblock {\em Chinese Ann. Math. Ser. B 23}, 2 (2002), 209--218.
\newblock Dedicated to the memory of Jacques-Louis Lions.

\bibitem{LiRao2003_sicon}
{\sc Li, T., and Rao, B.}
\newblock Exact boundary controllability for quasi-linear hyperbolic systems.
\newblock {\em SIAM J. Control Optim. 41}, 6 (2003), 1748--1755.

\bibitem{li2010no}
{\sc Li, T., Rao, B., and Wang, Z.}
\newblock Exact boundary controllability and observability for first order
  quasilinear hyperbolic systems with a kind of nonlocal boundary conditions.
\newblock {\em Discrete Contin. Dyn. Syst. 28}, 1 (2010), 243--257.

\bibitem{li2016book}
{\sc Li, T., Wang, K., and Gu, Q.}
\newblock {\em {E}xact {B}oundary {C}ontrollability of {N}odal {P}rofile for
  {Q}uasilinear {H}yperbolic {S}ystems}.
\newblock SpringerBriefs in Mathematics. Springer, Singapore, 2016.

\bibitem{Li_Duke85}
{\sc Li, T., and Yu, W.}
\newblock {\em {B}oundary {V}alue {P}roblems for {Q}uasilinear {H}yperbolic
  {S}ystems}.
\newblock Duke University Mathematics Series, V. Duke University, Mathematics
  Department, Durham, NC, 1985.

\bibitem{Zhuang2021}
{\sc Li, T., and Zhuang, K.}
\newblock A cut-off method to realize the exact boundary controllability of
  nodal profile for {S}aint-{V}enant systems on general networks with loops.
\newblock {\em J. Math. Pures Appl. 151\/} (2021), 1--27.

\bibitem{Macchelli2004Timo}
{\sc Macchelli, A., and Melchiorri, C.}
\newblock Modeling and control of the {T}imoshenko beam. {T}he distributed port
  {H}amiltonian approach.
\newblock {\em SIAM J. Control Optim. 43}, 2 (2004), 743--767.

\bibitem{Macchelli2007}
{\sc Macchelli, A., Melchiorri, C., and Stramigioli, S.}
\newblock Port-based modeling of a flexible link.
\newblock {\em IEEE Transactions on Robotics 23}, 4 (2007), 650--660.

\bibitem{Macchelli2009}
{\sc {Macchelli}, A., {Melchiorri}, C., and {Stramigioli}, S.}
\newblock Port-based modeling and simulation of mechanical systems with rigid
  and flexible links.
\newblock {\em IEEE Transactions on Robotics 25}, 5 (2009), 1016--1029.

\bibitem{Maschke1992}
{\sc Maschke, B., and {van der Schaft}, A.}
\newblock Port-controlled {H}amiltonian systems: Modelling origins and
  systemtheoretic properties.
\newblock {\em IFAC Proceedings Volumes 25}, 13 (1992), 359--365.

\bibitem{Munoz2020}
{\sc Muñoz-Simón, A., Wynn, A., and Palacios, R.}
\newblock Unsteady and three-dimensional aerodynamic effects on wind turbine
  rotor loads.
\newblock In {\em AIAA Scitech Forum\/} (2020).

\bibitem{Palacios2017modes}
{\sc Palacios, R.}
\newblock {\em Invariant manifolds in beam dynamics: free vibrations and
  nonlinear normal modes}.
\newblock Springer Berlin Heidelberg, 2017, pp.~1--8.

\bibitem{Palacios2011intrinsic}
{\sc Palacios, R., and Epureanu, B.}
\newblock An intrinsic description of the nonlinear aeroelasticity of very
  flexible wings.
\newblock In {\em 52nd AIAA/ASME/ASCE/AHS/ASC Structures, Structural Dynamics
  and Materials Conference\/} (2011).

\bibitem{Palacios2010aero}
{\sc Palacios, R., Murua, J., and Cook, R.}
\newblock Structural and aerodynamic models in nonlinear flight dynamics of
  very flexible aircraft.
\newblock {\em AIAA Journal 48}, 11 (2010), 2648--2659.

\bibitem{reissner1981finite}
{\sc Reissner, E.}
\newblock On finite deformations of space-curved beams.
\newblock {\em Zeitschrift f{\"u}r angewandte Mathematik und Physik ZAMP 32}, 6
  (1981), 734--744.

\bibitem{R2020}
{\sc Rodriguez, C.}
\newblock Networks of geometrically exact beams: well-posedness and
  stabilization.
\newblock {\em Math. Control Relat. Fields 0}, 0 (2021), 0--0.
\newblock Advance online publication.

\bibitem{RL2019}
{\sc Rodriguez, C., and Leugering, G.}
\newblock Boundary feedback stabilization for the intrinsic geometrically exact
  beam model.
\newblock {\em SIAM J. Control Optim. 58}, 6 (2020), 3533--3558.

\bibitem{Sarac2021}
{\sc Sarac, Y., and Zuazua, E.}
\newblock Sidewise control of 1-d waves, 2021.
\newblock arXiv preprint arXiv:2101.00473.

\bibitem{simo1985finite}
{\sc Simo, J.}
\newblock A finite strain beam formulation. {T}he three-dimensional dynamic
  problem. {P}art {I}.
\newblock {\em Comput. Methods in Appl. Mech. and Engrg. 49}, 1 (1985), 55 --
  70.

\bibitem{Simo1988}
{\sc Simo, J.~C., Marsden, J.~E., and Krishnaprasad, P.~S.}
\newblock The {H}amiltonian structure of nonlinear elasticity: the material and
  convective representations of solids, rods, and plates.
\newblock {\em Arch. Rational Mech. Anal. 104}, 2 (1988), 125--183.

\bibitem{strohm_dissert}
{\sc Strohmeyer, C.}
\newblock {\em Networks of nonlinear thin structures - theory and
  applications}.
\newblock PhD thesis, FAU University Press, 2018.

\bibitem{Schaft2002}
{\sc van~der Schaft, A.~J., and Maschke, B.~M.}
\newblock Hamiltonian formulation of distributed-parameter systems with
  boundary energy flow.
\newblock {\em J. Geom. Phys. 42}, 1-2 (2002), 166--194.

\bibitem{flotow_spacecraft}
{\sc von Flotow, A.~H.}
\newblock Traveling wave control for large spacecraft structures.
\newblock {\em Journal of Guidance, Control, and Dynamics 9}, 4 (1986),
  462--468.

\bibitem{vrabie2004}
{\sc Vrabie, I.~I.}
\newblock {\em {D}ifferential {E}quations: {A}n {I}ntroduction to {B}asic
  {C}oncepts, {R}esults, and {A}pplications}.
\newblock World Scientific, 2004.

\bibitem{kw2011}
{\sc Wang, K.}
\newblock Exact boundary controllability of nodal profile for 1-d quasilinear
  wave equations.
\newblock {\em Frontiers Math. China 6\/} (2011), 545--555.

\bibitem{kw2014}
{\sc Wang, K., and Gu, Q.}
\newblock Exact boundary controllability of nodal profile for quasilinear wave
  equations in a planar tree-like network of strings.
\newblock {\em Math. Methods Appl. Sci. 37\/} (2014), 1206--1218.

\bibitem{wang2014windturbine}
{\sc Wang, L., Liu, X., Renevier, N., Stables, M., and Hall, G.~M.}
\newblock Nonlinear aeroelastic modelling for wind turbine blades based on
  blade element momentum theory and geometrically exact beam theory.
\newblock {\em Energy 76\/} (2014), 487 -- 501.

\bibitem{YWang2021_NP_HUM}
{\sc Wang, Y., Leugering, G., and Li, T.}
\newblock {HUM} method to the exact boundary controllability of nodal profile
  for vibrating strings, 2021.
\newblock in preparation.

\bibitem{YWang2019partialNP}
{\sc Wang, Y., and Li, T.}
\newblock Exact boundary controllability of partial nodal profile for wave
  equations.
\newblock {\em Nonlinear Analysis: Real World Applications 62\/} (2021).

\bibitem{wang2006exact}
{\sc Wang, Z.}
\newblock Exact controllability for nonautonomous first order quasilinear
  hyperbolic systems.
\newblock {\em Chinese Ann. Math. Ser. B 27}, 6 (2006), 643--656.

\bibitem{weiss99}
{\sc Weiss, H.}
\newblock {\em Zur Dynamik geometrisch nichtlinearer Balken}.
\newblock PhD thesis, Technische Universit\"at Chemnitz, 1999.

\bibitem{Zhuang2018}
{\sc Zhuang, K., Leugering, G., and Li, T.}
\newblock Exact boundary controllability of nodal profile for {S}aint-{V}enant
  system on a network with loops.
\newblock {\em J. Math. Pures Appl. 129\/} (2018), 34--60.

\end{thebibliography}

\end{document}